\documentclass[a4paper, 10pt]{amsart}

\usepackage{amsmath}
\usepackage{amsthm}
\usepackage{amssymb}
\usepackage{esint}
\usepackage{verbatim}
\usepackage{tabularx}
\usepackage{dsfont}
\usepackage{enumerate}
\usepackage{hyperref}
\usepackage{empheq}
\usepackage[inner=2.4cm,outer=2.4cm,bottom=3.5cm,top=3.5cm]{geometry}
\usepackage{microtype}
\usepackage{subcaption}

\hypersetup{urlcolor=blue, colorlinks=true} 

\newtheorem{theorem}{Theorem}[section]

\newtheorem{lemma}[theorem]{Lemma}
\newtheorem{proposition}[theorem]{Proposition}
\newtheorem{corollary}[theorem]{Corollary}

\theoremstyle{remark}
\newtheorem{remark}[theorem]{Remark}
\theoremstyle{definition}

\numberwithin{equation}{section}

\newcommand{\R}{\ensuremath{\mathbb{R}}}
\newcommand{\N}{\ensuremath{\mathbb{N}}}
\newcommand{\Z}{\ensuremath{\mathbb{Z}}}
\newcommand{\veps}{\varepsilon}

\begin{document}

\title[Fujita exponent for discrete blow-up problems]{The Fujita exponent for finite difference approximations of nonlocal and local semilinear blow-up problems}

\author[F.~del~Teso]{F\'elix del Teso}

\address[F. del Teso]{Departamento de Matematicas, Universidad Aut\'onoma de Madrid (UAM), Campus de Cantoblanco, 28049 Madrid, Spain}

\email[]{felix.delteso\@@{}uam.es}

\urladdr{https://sites.google.com/view/felixdelteso}


\author[R. Ferreira]{Ra\'ul Ferreira}

\address[R. Ferreira]{Departamento de An\'alisis Matem\'atico y Matem\'atica Aplicada,                Universidad Complutense de Madrid, 28040,  Madrid, Spain.}
\email[]{raul\_ferreira@mat.ucm.es}
\urladdr{}



\begin{abstract}\noindent
We study monotone finite difference approximations for a broad class of reaction-diffusion problems, incorporating general symmetric L\'evy operators. By employing an adaptive time-stepping discretization, we derive the discrete Fujita critical exponent for these problems. Additionally, under general consistency assumptions, we establish the convergence of discrete blow-up times to their continuous counterparts. As complementary results, we also present the asymptotic-in-time behavior of discrete heat-type equations as well as an extensive analysis of discrete eigenvalue problems.
\end{abstract}

\maketitle



\section{Introduction and main results} \label{sec:intro}
The aim of this paper is to consider monotone finite difference approximations of nonnegative and nontrivial solutions for a large class of blow-up problems of the form
\begin{equation}\label{1.1}
\left\{
\begin{array}{ll}
w_t(x,t)=L w(x,t)+w^p(x,t), \qquad & x\in \mathbb R^N,\, t>0,\\
w(x,0)=\varphi, & x\in\mathbb R^N,
\end{array}\right.
\end{equation}
where the operator $L$ can be local, nonlocal, or both. More precisely, for a sufficiently smooth function $\phi:\R^N\to \R$, we define
\begin{equation}\label{eq:LevyOperator}
L \phi(x) =a \Delta \phi(x) + \textup{P.V.} \int_{|y|>0} (\phi(x+y)-\phi(x))d\mu(y),
\end{equation}
where $a\in \R_+$ and $\mu$ is a nonnegative symmetric Radon measure satisfying $\int_{|z|>0}  \min\{|z|^2, 1\}d\mu(z)<+\infty$.
In this work, we will only consider $p>1$ in order to have solutions with blow-up, that is, there exists a finite time $T$ such that
$$
\lim_{t\to T}\|w(\cdot,t)\|_{L^\infty(\R^d)} =\infty.
$$

The first study of the blow-up phenomenon was presented by Fujita in \cite{Fujita}, where the semilinear heat equation ($L = \Delta$) is analyzed. He proved that, for $1 < p < 1 + \frac{2}{N}$, every nontrivial nonnegative solution blows up in finite time, while, for $p > 1 + \frac{2}{N}$, there are both blow-up and global solutions. The values $p_0 = 1$ and $p_f = 1 + \frac{2}{N}$ are called the critical global exponent and the critical Fujita exponent, respectively. In the borderline case $p = p_f$, every positive solution blows up; see, for instance, \cite{Weiss}. For a general reference on blow-up in semilinear local problems, see \cite{QS}. For the study of the Fujita exponent in the nonlocal diffusion framework, we refer to the works \cite{Quiros-Jorge,Alfaro} for the case $d\mu(y) = J(y)dy$ with $J \in L^1(\mathbb{R}^N)$, and to \cite{NagasawaSirao69} for the fractional Laplacian operator, where $J(y) = C_{N,s} |y|^{-N-2s}$ with $s\in(0,1)$.

Numerical methods for local blow-up problems in a \emph{bounded domain} have also been extensively investigated both from theoretical and computational points of view. For example, in \cite{ADR,G,N}, the authors study fully discrete finite difference schemes. On the other hand, in \cite{ALM,GR,BB}, the authors consider a semi-discrete approximation in space (keeping $t$ continuous). In \cite{Budd}, an approximation based on the scaling of the equation is introduced.  However, for a numerical approximation of the blow-up problem with nonlocal diffusion, we have found only the unpublished article \cite{Kha} where the diffusion is given by the fractional laplacian, $L=-(-\Delta)^s$.

In the present present work, we consider always the \emph{unbounded domain} $\mathbb R^N$ and
we are interested in the analytical study of the finite difference approximation of problem \eqref{1.1}. To the best of our knowledge there are no theoretical studies (previous to the present work) of this type of problems neither in local or the nonlocal case. However, this kind of problems has been previously investigated from a computational point of view. In this case they need to restrict the problem to an artificial bounded ``computational" domain, and this requires the construction of an appropriate artificial boundary condition, see \cite{HW} for an introduction on artificial boundary methods and \cite{ZHH, QL} for the study of blow-up problem \eqref{1.1} with $L=\Delta$ in dimension $N=1,2$.

The \emph{main goals of this paper} are, firstly, to prove that a Fujita exponent exists for monotone finite difference approximations of \eqref{1.1} with adaptive time-stepping; and secondly, to show that, when blow-up occurs, the blow-up time of the numerical approximation converges to the blow-up time of the solution to \eqref{1.1}.

\textbf{Organization of the paper.} In Section \ref{sec:discprob}, we formulate the numerical scheme and present the fundamental assumptions regarding the discretizations. Sections \ref{sec:mainblowup} and \ref{sec:main-butimes} present the main results of the paper. Section \ref{sec:eigpair} investigates the eigenvalue problem for the discretization of the diffusion operator $L$. The asymptotic behavior of explicit finite difference approximations for heat-type equations with a fixed time step is discussed in Section \ref{comportamiento-asintotico-calor}. In Section \ref{bum-discreto}, we establish various conditions under which blow-up solutions of \eqref{eq-discreta} can be obtained. The Fujita exponent is derived in Section \ref{sec-fujita}. Under certain consistency assumptions on the diffusion operator, Section \ref{sec:convprop} proves the convergence of solutions (before the blow-up time) and convergence of the discrete to the continuous blow-up time. Finally, in Section~\ref{sec:numexp}, we present numerical experiments to validate and support our theoretical results.

\subsection{The discrete problem}\label{sec:discprob} Let us introduce now an explicit fully discrete approximation of problem \eqref{1.1}. Let $h>0$ be the spatial  discretization parameter and define $x_\alpha:=h \alpha \in h\Z^N$ with $\alpha \in \Z^N$. In order to capture the blow up phenomena with an explicit scheme, it is necessary to use a nonuniform time step. For now, let us consider a sequence of positive numbers  $\{\tau_j\}_{j=0}^\infty$ and define $t_0=0$ and $t_{j+1}=t_j + \tau_j$. We consider the following discrete blow-up problem,

\begin{equation}\label{eq-discreta}
\left\{
\begin{array}{ll}
\partial_{\tau_j} u(x_\alpha,t_j)=L_h u(x_\alpha,t_j)+u^p(x_\alpha,t_j),\quad & x_\alpha\in h\mathbb Z^N,\, t_j\in(0,T), \\
u(x_\alpha,0)=\varphi(x_\alpha), & x_\alpha\in h\mathbb Z^N,
\end{array}\right.
\end{equation}
where $\partial_{\tau_j}$ is the forward finite difference
\[
\partial_{\tau_j}u(x_\alpha,t_j)=\frac{u(x_\alpha,t_{j+1})-u(x_\alpha,t_j)}{\tau_j},
\]
and $L_h$ is the discrete diffusion operator
\begin{equation*}
L_h u(x_\alpha,t_j)= \sum_{\beta\in \Z^N\setminus\{0\}} \left(u(x_\alpha+x_\beta,t_j)-u(x_\alpha,t_j)\right) \omega(\beta,h)
\end{equation*}
with $\omega(\cdot,h)$ being a family of weights parameterized by $h$. The basic assumption on the weights is
\begin{equation}\label{as:omega1}\tag{\textup{A1}$_\omega$}
\textup{For every $h>0$,  $\omega(\alpha,h)=\omega(-\alpha,h)\geq0$ and $0<\|\omega(\cdot,h)\|_{\ell^1(\Z^N\setminus\{0\})}:= \sum_{\alpha \in \Z^N\setminus\{0\}} \omega(\alpha,h)<+\infty$.}
\end{equation}
\begin{remark}Note that, without loss of generality, we can always assume that $w(0,h)=0$ since this term will never be present in $L_h$.
\end{remark}
Moreover, we will also assume positivity of the first weights in the coordinate axis. More precisely, we will consider weights satisfying
\begin{equation}\label{as:omega2}\tag{\textup{A2}$_\omega$}
\textup{For every $h>0$, $\omega(e_i,h)>0$ for all $i=1,\ldots,N$.}
\end{equation}
Here $e_i\in \Z^N$ denotes the unit vector with value 1 at its coordinate $i$ and 0 in any other coordinate.  Finally, we will sometimes use the following assumption \begin{equation}\label{as:omega3}\tag{\textup{A3}$_\omega$}
\textup{For every $h>0$, $\omega(\cdot,h)$ is radially nonincreasing in $\Z^N\setminus\{0\}$.}
\end{equation}

\begin{remark}
We have chosen to work with an explicit scheme in \eqref{eq-discreta}, which, as will be shown later, requires a standard CFL-type stability condition that constrains the time discretization parameter in terms of the spatial one. Most of the results in this paper also hold for implicit schemes, which are usually not subject to such a condition. However, in order to accurately capture the blow-up behaviour of the solution in the context of blow-up problems such as the one considered here, the use of an adaptive time-stepping strategy becomes necessary even for implicit methods.  This adaptive constraint becomes more restrictive than the classical CFL condition after a few iterations, thereby eliminating any significant advantage of using implicit schemes. 
\end{remark}

Our running assumptions on the decay/growth of the weights are either
\begin{equation}\label{as:weight4}\tag{$\textup{A4}_\omega$}
   C_1 h^N |x_\alpha|^{-N-2s}\le \omega(\alpha,h)\leq C_2 h^N |x_\alpha|^{-N-2s} ,\quad \textup{for} \quad  |x_\alpha| \quad \textup{large enough,}
\end{equation}
or
\begin{equation}\label{as:weight5}\tag{$\textup{A5}_\omega$}
 M_2(h):= \sum_{\alpha\in \Z^N} \omega (\alpha,h)|x_\alpha|^2<+\infty.
\end{equation}

\begin{remark}\label{ejemplos}
We emphasize that the general assumptions on the weights introduced earlier are satisfied by many natural discretizations of both local and nonlocal operators.
\begin{enumerate}[$\bullet$]
\item \textbf{The Laplacian.} The classical central difference approximation of the Laplacian yields the weights $\omega(\alpha,h)=\frac{1}{h^2}$ for $|\alpha|=1$, and $\omega(\alpha,h)=0$ otherwise. In this case, assumptions \eqref{as:omega1}, \eqref{as:omega2}, \eqref{as:omega3}, and \eqref{as:weight5} are trivially satisfied.

\item \textbf{Nonlocal Lévy-type operators.} Consider the case $a=0$ and $d\mu(y)=J(y)\,dy$ in \eqref{eq:LevyOperator}, where $J: \R^N\setminus\{0\} \to \R^+$ is continuous on $\R^N \setminus \{0\}$, nonnegative, radially nonincreasing, and satisfies $J > 0$ in $B_\varepsilon(0)$ for some $\varepsilon > 0$. A finite difference discretization of $L$ can then be obtained (see \cite{dTEnJa18}) by choosing the weights
\[
\omega(\alpha,h)=\left\{
\begin{array}{ll}
h^N J(y_\alpha), &\text{if } |\alpha| \ne 0,\\
0, &\text{if } |\alpha| = 0.
\end{array}
\right.
\]
Clearly, \eqref{as:omega3} holds without further conditions. Moreover, if $h < \varepsilon$, then \eqref{as:omega2} is also satisfied. If, in addition, $J$ satisfies
\[
C_1 |x|^{-N-2s} \le J(x) \le C_2 |x|^{-N-2s},
\]
then \eqref{as:weight4} holds. On the other hand, if $\int_{\mathbb R^N} |x|^2 J(x) \,dx < \infty$, then \eqref{as:weight5} holds. Finally, since $J$ is radially nonincreasing, we have
\[
\sum_{\alpha \in \Z^N\setminus\{0\}} \omega(\alpha,h)=\sum_{\alpha\in \Z^N\setminus\{0\}} h^N J(y_\alpha)\le   C \left(h^N J(e_1)  + \int_{\mathbb R^N\setminus B_h(0)} J(y) \,dy\right) <\infty,
\]
and hence \eqref{as:omega1} is also satisfied.

Several prototypical Lévy operators satisfy these assumptions:
\begin{enumerate}[\rm (a)]
    \item \emph{The fractional Laplacian $-(-\Delta)^s$ for $s\in(0,1)$}. In this case, $J(y) = c_{N,s} |y|^{-N-2s}$ for some constant $c_{N,s}>0$.
    
    \item \emph{$s$-stable operators for $s\in(0,1)$}. Here $J$ is continuous in $\R^N \setminus \{0\}$, radially nonincreasing, and satisfies $c_1 |y|^{-N-2s} \le J(y) \le c_2 |y|^{-N-2s}$ for some positive constants $c_1$, $c_2$.
    
    \item \emph{Zero-order operators}. These correspond to $J \in C(\mathbb R^N) \cap L^1(\mathbb R^N)$, radially nonincreasing, and such that $J > 0$ in $B_\varepsilon(0)$. Note that $J$ can be compactly supported.
    
    \item \emph{The relativistic Schrödinger operator $-(I - \Delta)^s + I$ for $s\in(0,1)$}. This operator has a singular integral representation with $J(x) \sim |x|^{-N-2s}$ as $|x| \to 0$ and exponential decay as $|x| \to \infty$ (see e.g. \cite{FallFelli2015,Roncal2023}). In this case, the condition $\int_{\R^N} |y|^2 J(y) \,dy < \infty$ is satisfied. 

    \item \emph{Discrete operators: } Operators of the form of \eqref{eq:LevyOperator} with $\mu$ being a linear combination of Dirac deltas. For example, $d\mu(z)=d\delta_1(z)+d\delta_{1}(z)$ yields $Lu(x)=\phi(x+1)+\phi(x-1)-2\phi(x)$. Other nonlocal discrete operators can be obtained in a similar way.
\end{enumerate}

\item \textbf{Mixed local-nonlocal operators.} One can also consider combinations of the form $L = b_1 L_1 + b_2 L_2$, where $b_1, b_2 > 0$, and $L_1, L_2$ are any of the operators mentioned above. In this case we can choose the weights for $L$ as the sum of the two corresponding weights for $L_1$ and $L_2$.
\end{enumerate}
\end{remark}

In order to ensure good properties for the solution of \eqref{eq-discreta} we need a standard CFL-type stability condition coupled with a blow-up correction. More precisely, we will consider the following time-stepping discretization:
\begin{equation}\label{CFL1}\tag{CFL}
\textup{Let} \,\,\, \tau>0 \,\,\, \textup{such that} \quad  \tau \leq\frac{1}{4\sum_{\alpha \in \Z^N\setminus\{0\}} \omega(\alpha,h)}, \quad \textup{and define} \quad \tau_j:= \tau \min\{1, \|u(\cdot,t_j)\|_{\ell^\infty(h\Z^N)}^{1-p}\}.
\end{equation}
The choice of $\tau$ is the standard one that ensures comparison principle (see e.g. \cite{dTEnJa19}), while the one for $\tau_j$ is based on the adaptive time step of the associated ODE problem (see Appendix \ref{sec:choicetau} for details).

We refer to \cite{G,N} for different choices of time steps for approximations of the semilinear heat equation in a bounded domain.

\subsection{Blow-up results for the discrete problem}\label{sec:mainblowup}
The first main objective of this work is to study the blow-up Fujita exponent for  problem \eqref{eq-discreta}. We say that a solution \emph{blows up} if there exists a time $T_{h,\tau}:=\sum_{j=0}^\infty \tau_j<\infty$ such that
$$
t_j \to T_{h,\tau} \quad\mbox{and}\quad \|u(\cdot,t_j)\|_\infty\to \infty \quad\mbox{as}\quad j\to\infty.
$$
Otherwise, we say that the solution is \emph{global in time}.

For some reaction-diffusion problems, an heuristic argument to obtain the Fujita exponent is given by equating the decay exponent of the pure diffusion equation with the ODE blow-up exponent, see \cite{DengLevine00}. In the semilinear heat equation, the diffusion decay rate is given by the decay of the Gauss kernel, $\rho_d= N/2$, while the solutions of the corresponding ODE ($y'(t)=y(t)^p$) are  given by
\begin{align*}
 y(t)= (y(0)^{1-p}-(p-1) t)^{-\rho_r}, \quad \textup{where} \quad \rho_r=1/(p-1)
\end{align*}
Thus $\rho_d=\rho_r$ gives $p_f=1+2/N$.
The same can be applied for the diffusion nonlocal problem. For instance, the fractional Gauss kernel decays like $t^{-N/(2s)}$ and the corresponding ODE  exponent $\rho_r=1/(p-1)$ produces $p_f=1+2s/N$.

We will show that this heuristic argument also works for problem \eqref{eq-discreta}. For the discrete ODE, as we show in Appendix \ref{sec:choicetau}, we have that $\rho_r=1/(p-1)$, so we need to know the exponent decay of the linear problem
\begin{equation}\label{eq:discheatintro}
\left\{
\begin{array}{ll}
\partial_{\tau} z(x_i,t_j)=L_h z(x_i,t_j), \quad & x_i\in h\mathbb Z^N,\\
z(x_i,0)=\varphi(x_i).
\end{array}\right.
\end{equation}
Observe that the time step for the pure diffusion problem \eqref{eq:discheatintro} is constant.

\begin{theorem}\label{comportamiento lineal}
Assume \eqref{as:omega1},
$\tau \leq(4\sum_{\alpha \in \Z^N\setminus\{0\}}\omega(\alpha,h))^{-1}$ and either \eqref{as:weight4} and $s\in(0,1)$ or \eqref{as:weight5} and $s=1$.
Let $z$ be the solution of \eqref{eq:discheatintro} with initial datum $\varphi\in \ell^1(h\Z^N)$.
Then,
\begin{equation}\label{eq:asbehaintro}
    \limsup_{j\to\infty} \,t_j^{\frac{N}{2s}} \sup_{\alpha \in \Z^N}|z(x_\alpha,t_j)-\|\varphi\|_{\ell^1(h\Z^N)}\Gamma_s(x_\alpha,t_j)|\le C,
\end{equation}
where $\Gamma_s$ is the fundamental solution of the fractional heat equation (the heat equation if $s=1$)
$
(\Gamma_s)_t=-K (-\Delta)^{s} \Gamma_s,
$
for some constant $K>0$ (see Section \ref{sec:symfour} for details). In particular,
\[
\|z(\cdot,t_j)\|_{\ell^\infty(h\Z^N)}\leq \tilde{C} t^{-\frac{N}{2s}}.
\]
\end{theorem}
With the decay  of the linear equation, $\rho_d=\rho_r$ implies that the Fujita exponent is $p_f= 1+ 2s/N$. We have the following result for exponents above and below the Fujita exponent:
\begin{theorem}\label{teo.fujita1}
    Assume \eqref{CFL1}, \eqref{as:omega1}, \eqref{as:omega2}, \eqref{as:omega3} and either \eqref{as:weight4}  and $s\in(0,1)$ or \eqref{as:weight5} and $s=1$. Then
    \begin{enumerate}[\rm(a)]
    \item\label{teo.fujita1-blowup} if $1<p\le  1+2s/N$, then all solutions of \eqref{eq-discreta} blow up in finite time;
    \item\label{teo.fujita1-global} if $p>1+2s/N$, then there exist both global and blow-up solutions of \eqref{eq-discreta}.
    \end{enumerate}
\end{theorem}

\begin{remark}
Let us assume that \eqref{as:weight4} holds in the the previous Theorem.
For the existence of global solution if $p>1+2s/N$ we only need the lower estimate, $C_1 h^N |x_\alpha|^{-N-2s}\le  \omega(\alpha,h)$. However, in the blow-up range if $1<p<1+2s/N$ we only need the opposite inequality, $ \omega(\alpha,h)\leq C_2 h^N |x_\alpha|^{-N-2s}$ while for the critical case $p=1+2s/N$ we need both.
\end{remark}

Next, we turn our attention to the blow-up rate of the discrete problem.

\begin{theorem}\label{tasas}
     Assume \eqref{as:omega1} and \eqref{CFL1}. Let $u$ be a blow-up solution of \eqref{eq-discreta} with blow-up time $T_{h,\tau}$. Then, there exist two positive constants $C_1$ and $C_2$ (independent on $h$ and $\tau$) such that
$$
C_1 (T_{h,\tau}-t_j)^{\frac{-1}{p-1}}\le \|u(\cdot,t_j)\|_{\ell^\infty(h\Z^N)}\le
C_2 (T_{h,\tau}-t_j)^{\frac{-1}{p-1}},
$$
for all $t_j\ge t_m$, with $m=m(h,\tau)$.
\end{theorem}

\subsection{Convergence of blow-up times}\label{sec:main-butimes}
In order to ensure convergence, we need to assume some consistency properties of the operators and some regularity of the solutions.

Given $\overline{T}>0$, we consider a normed space $\mathcal{H}_{\overline{T}}$ of functions $f:\R^d\times[0,\overline{T}]\to \mathbb{R}^N$ such that
\begin{equation}\label{Lambdas}
 \Lambda_1(x_\alpha,t_j):= L f(x_\alpha,t_j) - L_h f(x_\alpha,t_j) \quad  \textup{and} \quad \Lambda_2(x_\alpha,t_j):= f_t(x_\alpha,t_j) - \partial_{\tau_j} f(x_\alpha,t_j)
\end{equation}
satisfy
\begin{equation}\label{as:consistency} \tag{$\textup{A}_{\textup{c}}$}
     \sup_{x_\alpha\in h\Z^N,\, t_j \in [0,\overline{T}]  }|\Lambda_1(x_\alpha,t_j)| = \|f\|_{\mathcal{H}_{\overline{T}}} \varrho_1(h)  \quad  \textup{and} \quad  \sup_{x_\alpha\in h\Z^N,\, t_j \in [0,\overline{T}]  }|\Lambda_2(x_\alpha,t_j)| = \|f\|_{\mathcal{H}_{\overline{T}}} \varrho_2(\tau)
\end{equation}
for some modulus of consistency $\varrho_i$ such that $\varrho_i(\xi)\to0$ as $\xi\to0^+$ for $i=1,2$.

\begin{remark}
For instance, considering  the classical  central difference approximation  to approximate the heat equation, it is well know that if $f\in C_b^{2+\varepsilon,1+\delta}(\mathbb R^N\times [0,\overline T])$ for some $\varepsilon>0$ and $\delta>0$, we obtain consistency in space ($\rho(h)=Ch^\varepsilon$) and in time ($\mu(\tau)=C\tau^\delta$). Then, we can take
$$
\mathcal{H}_{\overline{T}}= C_b^{2+\varepsilon,1+\delta}(\mathbb R^N\times [0,\overline T]).
$$

If we approximate the fractional heat equation using the weights given by
Remark \ref{ejemplos}, we can take $
\mathcal{H}_{\overline{T}}= C_b^{2s+\varepsilon,1+\delta}(\mathbb R^N\times [0,\overline T]).
$
\end{remark}

We present first the convergence result of the solution before the blow-up time.

\begin{theorem}\label{thm:convergencebeforeblowup}
    Assume \eqref{as:omega1}, \eqref{as:consistency} and \eqref{CFL1}. Let $w$ be a blow-up solution of \eqref{1.1} with blow-up time $T$. Assume $w\in \mathcal{H}_{T-\rho}$ for all $\rho\in(0,T)$. Let also $u$ be a blow-up solution of \eqref{eq-discreta} with  blow-up  time $T_{h,\tau}$. Then,
    \[
    \max_{x_\alpha\in h\mathbb{Z}^N, t_j \in [0,T-\rho]} \left|w(x_\alpha,t_j)-u(x_\alpha,t_j)\right| \leq C (\varrho_1(h)+\varrho_2(\tau)) \quad \textup{as} \quad h,\tau\to0^+,
    \]
    with $C=C(p,T,\rho, w)>0$. 
     In particular, $
    \liminf_{h,\tau\to0} T_{h,\tau} \geq T$.

\end{theorem}
To prove convergence of the blow-up times, we need to restrict ourselves  either to a smaller class of diffusion operators (zero-order ones) or to special types of initial conditions.

\begin{theorem}\label{thm:conbwtimes}.
    Let the assumptions of Theorem \ref{thm:convergencebeforeblowup} hold. Additionally assume that at least one of the following properties holds:
\begin{enumerate}[\rm (a)]
    \item\label{lemconbwt-item1} There exists $h_0\in(0,1)$ such that $\sup_{h\in(0,h_0)}\sum_{\beta\in \mathbb{Z}^N\setminus\{0\}} \omega(\beta,h)<+\infty$.
\item\label{lemconbwt-item2} There exist $\veps,h_0\in(0,1)$ such that the initial data $\varphi$ satisfies $
L_h \varphi(x_\alpha) + (1-\veps) \varphi^p(x_\alpha)\geq 0$  for all   $h\in(0,h_0)$  and  $x_\alpha\in h\mathbb{Z}^N$.
\end{enumerate}
Then, $
\lim_{h,\tau\to0^+} T_{h,\tau}=T$.
\end{theorem}
\begin{remark}\label{rem:datagood}
Let us briefly comment on the additional assumptions of Theorem \ref{thm:conbwtimes}.
Condition \eqref{lemconbwt-item1} appears in a very natural way when you approximate a
nonlocal operator of order zero of the form $d\mu(y)=J(y)dy$ with $J\in L^1(\R^N)$. Indeed, considering the weights given in  Remark \ref{ejemplos}, we get that
$
\sum_{\beta\in h\mathbb{Z}^N} \omega(\beta,h) \le \int_{\R^N} J(y)dy <\infty.
$
On the other hand, it is simple to construct initial conditions satisfying \eqref{lemconbwt-item2}. For example, we can consider $u_0\in C^\infty_c(\R^N)$ a nonnegative radially decreasing function supported in $B_R(0)$ with $\partial^4_r u_0\ge0$ for  $|x|>R-\delta$ and $L u_0(x)+(1-2\varepsilon)u_0^p(x)\ge0$ for $x\in\R^N$. A complete proof of this fact can be found at the end of Section \ref{sec:convprop}. \end{remark}

\section{Eigenvalue problem}\label{sec:eigpair}
As before, let $h>0$ be a discretization parameter and define $x_\alpha:=h \alpha \in h\Z^N$ with $\alpha \in \Z^N$. We recall that, for a function $\phi: h\Z^N\to \R$, we define the discrete diffusion  operator
\begin{equation}\label{eq:operator}
L_h\phi(x_\alpha)= \sum_{\beta\not=0} \left(\phi(x_\alpha+x_\beta)-\phi(x_\alpha)\right) \omega(\beta,h).
\end{equation}

Now, given a bounded domain $\Omega$ we consider the Dirichlet eigenvalue problem given by
\begin{equation}\label{eq:disceig}
\left\{
\begin{split}
- L_h \phi (x_\alpha) &= \lambda_{h,\Omega} \phi(x_\alpha), \quad \textup{if} \quad x_\alpha\in \Omega,\\
\phi (x_\alpha) &= 0, \hspace{1.74cm} \textup{if} \quad x_\alpha\in \R^N \setminus \Omega.
\end{split}\right.
\end{equation}
The problem is posed in the functional space
\[
X_h(\Omega):= \{\phi\in \ell^2(h\Z^N)\,:\, \phi(x_\alpha)=0 \quad \textup{if} \quad x_\alpha \in h\Z^N\setminus \Omega\}
\]
endowed with the inner product and norm
\[
\langle\phi_1,\phi_2\rangle_{ \ell^2(h\Z^N)}:= h^N \sum_{x_\alpha\in h\Z^N} \phi_1(x_\alpha)\phi_2(x_\alpha) \quad  \textup{and}\quad
\|\phi\|_{ \ell^2(h\Z^N)}=\langle\phi_1,\phi_2\rangle_{ \ell^2(h\Z^N)}^\frac{1}{2}.
\]
We have the following regarding the eigenpairs of \eqref{eq:disceig}.

\begin{proposition}\label{prop:eig}
Assume \eqref{as:omega1}. Then, the first eigenvalue of \eqref{eq:disceig}  is given by
\[
\lambda_{h,\Omega}= \min_{{\phi\in X_h(\Omega)}\atop{\|\phi\|_{\ell^2(h\Z^N)}=1}} \langle\phi,-L_h\phi\rangle_{ \ell^2(h\Z^N)}.
\]
Moreover,
\begin{enumerate}[\rm (a)]
\item\label{prop:eig-item1} $\lambda_{h,\Omega}$ is strictly positive and there exists a nonnegative eigenfunction.
\item \label{prop:eig-item3} If \eqref{as:omega2} holds, then  any   nontrivial eigenfunction associated to $\lambda_{h,\Omega}$ is either strictly positive  or strictly negative in $ h\Z^N\cap \Omega$. Moreover, $\lambda_{h,\Omega}$ is simple.
\end{enumerate}
\end{proposition}

\begin{proof}
Note that, since $\Omega$ is bounded, there exists a finite number of points $x_\alpha\in \Omega$. Let us label them by $x_{\alpha_i}$ for $i=1,\ldots,I$ for some $I>0$. Then, problem \eqref{eq:disceig} is equivalent to
\[
\phi(x_{\alpha_i})  \|\omega(\cdot,h)\|_{\ell^1(\Z^N)} -  \sum_{{j=1}\atop{j\not=i}}^{I} \phi(x_{\alpha_j}) \omega(\alpha_i-\alpha_j,h)= \lambda_{h,\Omega} \phi(x_{\alpha_i}) \quad \textup{for all} \quad i=1,\ldots,I.
\]
That is, we are looking for the eigenpairs of the matrix $M\in \R^{I\times I}$ given by
\begin{equation}\label{eq:Meig}
M_{i,j} =\left\{
\begin{split}
  \|\omega(\cdot,h)\|_{\ell^1(\Z^N)}  \quad &\textup{if} \quad i=j\\
  -\omega(\alpha_i-\alpha_j,h) \quad &\textup{if} \quad i\not=j.
\end{split}
\right.
\end{equation}
Note that, by symmetry of the weights, $M$ is symmetric, and the first eigenvalue is given by
$
\lambda_{h,\Omega}= \min_{\|w\|=1} \{w^T M w\}.
$
Moreover, given any $w\in \R^I$, define $\phi:h\Z^n\to \R$ by $\phi(x_{\alpha_i})=w_i$ for $i=1,\ldots,I$ and $\phi=0$ otherwise. Then
\[
\begin{split}
w^T M w =  \sum_{i=1}^I w_i (Mw)_i = \sum_{i=1}^I \phi(x_{\alpha_i}) (-L_h \phi(x_{\alpha_i})) = \sum_{x_\alpha\in h \Z^N}  \phi(x_{\alpha}) (-L_h \phi(x_{\alpha}))= \frac{\langle\phi,-L_h\phi\rangle_{ \ell^2(h\Z^N)}}{h^N},
\end{split}
\]
and $
w^T w = \sum_{x_\alpha\in h \Z^N}  \phi(x_{\alpha}) \phi(x_{\alpha}) =\|\phi\|_{ \ell^2(h\Z^N)}^2/h^N$.
Thus,
\[
\lambda_{h,\Omega}= \min_{\|w\|=1} \{w^T M w\}= \min_{w\not=0}\frac{ w^T M w}{w^T w}= \min_{{\phi\in X_h(\Omega)}\atop{\phi\not=0}}\frac{ \langle\phi,-L_h\phi\rangle_{ \ell^2(h\Z^N)}}{\|\phi\|_{ \ell^2(h\Z^N)}^2}= \min_{{\phi\in X_h(\Omega)}\atop{\|\phi\|_{ \ell^2(h\Z^N)}=1}} \langle\phi,-L_h\phi\rangle_{ \ell^2(h\Z^N)}.
\]
We will check now that $\lambda_{h,\Omega} \geq0$. This is a consequence of the fact that $M$ is weakly diagonally dominant, and the elements of the diagonal are nonnegative (i.e. $M$ is positive semi-definite). Thus $w^T M w\geq0$ for all $w\in \R^I$. Once this is established, we show now that $\lambda_{h,\Omega} >0$. Assume by contradiction that $\lambda_{h,\Omega}=0$  and let $\phi$ a nontrivial eigenvector.  Then, $L_h \phi (x_\alpha)=0$ for all $x_\alpha \in \Omega$. In particular, take $\overline{\alpha}$ such that
$
\phi(x_{\overline{\alpha}})= \max_{x_{\alpha}\in h\Z^N} \phi(x_\alpha)=:\mathcal{M}$.
Since $\|\omega(\cdot,h)\|_{\ell^1(\Z^N\setminus\{0\})}>0$, there exists $\gamma\not=0$ such that, $\omega(\gamma,h)>0$. Then,
\[
0 = L_h\phi(x_{\overline{\alpha}})= \sum_{\beta\in\Z^N\setminus\{0\}} \left(\phi(x_{\overline{\alpha}}+x_\beta)-\phi(x_{\overline{\alpha}})\right) \omega(\beta,h) \leq \left(\phi(x_{\overline{\alpha}}+x_\gamma)-\phi(x_{\overline{\alpha}})\right) \omega(\gamma,h) .
\]
In particular, this shows that $\phi(x_{\overline{\alpha}}+x_\gamma)=\mathcal{M}$. If $x_{\overline{\alpha}}+x_\gamma \in \R^N \setminus \Omega$, we have shown that $\mathcal{M}=0$. Otherwise, we iterate this argument $n\in \N$ times in such a way that $x_{\overline{\alpha}}+x_{n\gamma} \in \R^N \setminus \Omega$, and the same conclusion follows. In a similar way, we can conclude that $\min_{x_{\alpha}\in h\Z^N} \phi(x_\alpha)=0$, and thus $\phi=0$ in $h\Z^N$. This is a contradiction with the fact that $\|\phi\|_{\ell^2(h\Z^N)}=1$.

On the other hand, given any vector $w$ with $\|w\|=1$ we have that
\[
w^T M w= \sum_{i,j=1}^I w_i M_{i,j} w_j= \sum_{i=1}^I  M_{i,i} w_i^2- \sum_{{i,j=1}\atop{j\not=i}}^{I} w_i |M_{i,j}|w_j \geq \sum_{i=1}^I  M_{i,i} w_i^2- \sum_{{i,j=1}\atop{j\not=i}}^{I} |w_i| |M_{i,j}||w_j | = \tilde{w}^T M \tilde{w}
\]
with $\tilde{w}_i=|w_i|$. Thus \eqref{prop:eig-item1} holds.

Now we prove \eqref{prop:eig-item3}.
Assume, by contradiction, that there exists $x_\alpha\in \Omega$ such that a nontrivial nonnegative eigenfunction is such that $\phi(x_\alpha)=0$. Then,
\[
0= \lambda_{h,\Omega} \phi(x_\alpha) =\sum_{\beta\in\Z^N\setminus\{0\}} \left(\phi(x_\alpha+x_\beta)-\phi(x_\alpha)\right) \omega(\beta,h) = \sum_{\beta\in\Z^N\setminus\{0\}} \phi(x_\alpha+x_\beta) \omega(\beta,h).
\]
Since $\phi\geq0$ and $\omega(\beta,h)\geq0$, the above identity implies that $\phi(x_\alpha+x_\beta)=0$ for all $\beta\in \Z^N$ such that $\omega(\beta,h)>0$. In particular, by assumption \eqref{as:omega2}, $\omega(\pm e_i,h)>0$ for $i=1,\ldots,N$, and thus  $\phi(x_\alpha\pm e_i)=0$. Repeating we can show that $\phi(x_\alpha)=0$ for every $x_\alpha\in \Omega$, which is a contradiction with the fact that $\phi$ was a nontrivial eigenfunction. This implies that the nontrivial eigenfunction associated to $\lambda_{h,\Omega}$ is either strictly positive  or strictly negative in $ h\Z^N\cap \Omega$. Indeed, if we assume that there exist $x_{\alpha}, x_\beta\in\Omega$ such that $\phi(x_{\alpha})<0<\phi(x_{\beta})$, then the eigenfunction given by
$
\psi=| \phi-|\phi||
$
is nonnegative, nontrivial ($\psi(x_\alpha)>0$) and vanishes at $x_\beta$. This is a contradiction.

Finally, we prove that $\lambda_{h,\Omega}$ is simple. Assume that there are two linearly independent eigenvectors $v$ and $\tilde{v}$, then any $\overline{v}\in \textup{span}\{v,\tilde{v}\}$ is also an eigenvector. But since  $\textup{span}\{v,\tilde{v}\}$ is a subspace of dimension 2, then there must be some $\overline{v}\in \textup{span}\{v,\tilde{v}\}$ with components of different signs, which is a contradiction.
\end{proof}

\begin{remark}
Assumption \eqref{as:omega2} is necessary to prove \eqref{prop:eig-item3}. Indeed, in dimension $N=1$ we take $\omega(\pm 2,1)=1$  and $\omega(\alpha,1)=0$ otherwise and consider the eigenvalue problem for $\Omega=(-2,3)$, the associated matrix $M$ as in the proof above, and its eigenvectors
$$
M=\left(
\begin{smallmatrix}
2 &0& -1& 0\\
0 &2 &0 &-1\\
-1& 0& 2 &0\\
0 &-1& 0& 2
\end{smallmatrix}\right),\quad
v_1=\left(
\begin{smallmatrix}
1\\ 0\\ 1\\0
\end{smallmatrix}\right), \quad
v_2=\left(
\begin{smallmatrix}
0\\ 1\\ 0\\1
\end{smallmatrix}\right), \quad
v_3=\left(
\begin{smallmatrix}
-1\\ 0\\ 1\\0
\end{smallmatrix}\right) \quad  \textup{and} \quad
v_4=\left(
\begin{smallmatrix}
0\\ -1\\ 0\\1
\end{smallmatrix}\right)
$$
together with their correspondent positive eigenvalues
$\lambda_1=\lambda_2=1$, and $\lambda_3=\lambda_4=3$.
Note that $\lambda_1=1$ is a double eigenvalue and $v_1$ and $v_2$ are not strictly positive eigenvectors.
\end{remark}

\begin{remark}
Note that we can express the eigenvalue $\lambda_{h,\Omega}$ in the following equivalent forms:
\begin{equation}\label{eq:formeigdisc}
\begin{split}
\lambda_{h,\Omega}&= \min_{{\phi\in X_h(\Omega)}\atop{\|\phi\|_{\ell^2(h\Z^N)}=1}} \langle\phi,-L_h\phi\rangle_{ \ell^2(h\Z^N)}=\frac{h^N}{2} \min_{{\phi\in X_h(\Omega)}\atop{\|\phi\|_{\ell^2(h\Z^N)}=1}} \sum_{x_\alpha\in h \Z^n}  \sum_{x_\beta \in h\Z^N} (\phi(x_\alpha)-\phi(x_\beta))^2 \omega(\alpha-\beta,h)\\
&= \|\omega(\cdot,h)\|_{\ell^1(\Z^N)} - h^N \max_{{\phi\in X_h(\Omega)}\atop{\|\phi\|_{\ell^2(h\Z^N)}=1}} \sum_{x_\alpha\in \Omega} \sum_{x_\beta\in \Omega} \phi(x_\alpha)\phi(x_\beta) \omega (\alpha-\beta,h).
\end{split}
\end{equation}
\end{remark}

We will study now the behaviour of the first eigenvalue for dilation domains.

\begin{lemma}\label{lem:groeigRsecond}
Assume  \eqref{as:omega1}.
For $R>0$, define the set
$
\Omega_R:=R\Omega=\{Rx\in \R^N\, : \, x\in \Omega\}.
$
Then, for $R$ large enough, we have
\[
\lambda_{h,\Omega_R} \leq C M_2(h) R^{-2},
\]
where $M_2(h)$ is the second moment of the weights and $C$ is a constant independent  on $h$ and $R$.
\end{lemma}
\begin{proof}
Consider a function $\psi \in C^1_c(\R^N)$ supported in $\Omega$ and let $L$ be the Lipschitz constant of $\psi$. Define the family of functions $\phi_R : h\Z^N\to \R$ given by
\[
\phi_R(x_\alpha)= \psi(x_\alpha/R)/\|\psi(\cdot/R)\|_{\ell^2(h\Z^N)}.
\]
Note that $\phi_R \in X_h(\Omega_R)$, $\|\phi_R(\cdot/R)\|_{\ell^2(h\Z^N)}=1$, and
$|\phi_R(x_\alpha)-\phi_R(x_\beta)| \leq \frac{L}{\|\psi(\cdot/R)\|_{\ell^2(h\Z^N)}}\frac{|x_\alpha-x_\beta|}{R}$.
Then, using the characterization of $\lambda_{h,\Omega_R}$ given by \eqref{eq:formeigdisc} and the symmetry of the weights, we get
\[
\begin{split}
\lambda_{h,\Omega_R} &\leq \frac{h^N}{2}\sum_{x_\alpha\in h \Z^N}  \sum_{x_\beta \in h\Z^N} (\phi_R(x_\alpha)-\phi_R(x_\beta))^2 \omega(\alpha-\beta,h)\\
&\leq h^N\sum_{x_\alpha\in \Omega_R\cap h \Z^N  }  \sum_{x_\beta \in h\Z^N} (\phi_R(x_\alpha)-\phi_R(x_\beta))^2 \omega(\alpha-\beta,h)\\
&\leq  \frac{M_2(h) L^2}{R^2}  \frac{1}{\|\psi(\cdot/R)\|_{\ell^2(h\Z^N)}^2}\sum_{x_\alpha\in \Omega_R\cap h \Z^N  } h^N \leq  \frac{M_2(h) L^2}{R^2} \frac{|\Omega|+ o_{\frac{h}{R}}(1)}{\|\psi\|^2_{L^2(\Omega)}+ o_{\frac{h}{R}}(1)},
\end{split}
\]
where $o_{\frac{h}{R}}(1) \to 0$ as $R\to+\infty$ for $h>0$ fixed. This concludes the proof.
\end{proof}

The following result is related to weights with fractional moments.
\begin{lemma}\label{lem:decayeigfrac}
Assume  \eqref{as:omega1}. Moreover, assume that
$
\omega(\alpha,h)\leq C h^N |x_\alpha|^{-N-2s}
$
for some $s\in(0,1)$ and some $C$ depending only on $N$ and $s$.
Then, for $R$ large enough, we have
$
\lambda_{h,\Omega_R} \leq \tilde{C} R^{-2s}
$
where $\tilde{C}$ is a constant depending only on $\Omega$, $N$ and $s$.
\end{lemma}
\begin{proof}
Let $L_h$ be defined by \eqref{eq:operator} with weights $\omega(\alpha,h)= h^N|x_\alpha|^{-N-2s}$ if $\alpha \not=0$ and $\omega(0,h)=0$. Clearly, assumptions \eqref{as:omega1} and \eqref{as:omega2} are satisfied in this case. Consider $\phi$ to be a solution of the eigenvalue problem \eqref{eq:disceig} with grid size $h/R$, i.e,
\begin{equation}
\left\{
\begin{split}
- L_{\frac{h}{R}} \phi ({\frac{h}{R}} \alpha) &= \lambda_{{\frac{h}{R}},\Omega} \phi({\frac{h}{R}}\alpha), \quad \textup{if} \quad {\frac{h}{R}}\alpha\in \Omega,\\
\phi ({\frac{h}{R}} \alpha) &= 0, \hspace{1.84cm} \textup{if} \quad {\frac{h}{R}}\alpha\in \R^N \setminus \Omega.
\end{split}\right.
\end{equation}
Now define $\tilde{\phi} : h\Z^N \to R$ given by $\tilde{\phi}(h\alpha)= \phi ({\frac{h}{R}} \alpha)$. Clearly, $\tilde{\phi}(h \alpha)=0$ if $h \alpha \in \R^N \setminus \Omega_R$, and
\[
\begin{split}
L_h \tilde{\phi}(h \alpha) &=  \sum_{\beta \not=0} \Big(\tilde{\phi}(h\alpha+h\beta)-\tilde{\phi}(h\alpha)\Big) \frac{h^N}{|h\alpha|^{N+2s}} = \frac{1}{R^{2s}}  \sum_{\beta \not=0} \Big(\phi({\frac{h}{R}}\alpha+{\frac{h}{R}}\beta)-\phi({\frac{h}{R}}\alpha)\Big) \frac{1}{|{\frac{h}{R}}\alpha|^{N+2s}}\\
&= - \frac{\lambda_{{\frac{h}{R}},\Omega}}{R^{2s}} \phi({\frac{h}{R}}\alpha)=- \frac{\lambda_{{\frac{h}{R}},\Omega}}{R^{2s}} \tilde{\phi}(h \alpha),
\end{split}
\]
for all $h\alpha \in \Omega_R$. Thus, $\lambda_{h,\Omega_R}=\frac{\lambda_{{\frac{h}{R}},\Omega}}{R^{2s}}$. It remains to show  that $\lambda_{{\frac{h}{R}},\Omega}$ is uniformly bounded, which will conclude the proof.
For simplicity, let us denote by $\tilde{h}=h/R$. Take $\psi$ as in the proof of Lemma \ref{lem:groeigRsecond} and let
\[
\tilde{\psi}(\tilde{h}\alpha)= \psi(\tilde{h}\alpha)/\|\psi\|_{\ell^2(\tilde{h}\Z^N)}.
\]
Note that
$
|\tilde{\psi}(\tilde{h}\alpha)-\tilde{\psi}(\tilde{h}\beta)| \leq C \frac{L}{\|\psi\|_{\ell^2(\tilde{h}\Z^N)}}\max\{\tilde{h}|\alpha-\beta|,1\}.
$
with $C=\max\left\{\|\psi\|_{L^\infty(\R^N)},1\right\}$. Denote by $\tilde{x}_\alpha=\tilde{h}\alpha$ for all $\alpha \in \R^N$. Then,
\[
\begin{split}
\lambda_{{\frac{h}{R}},\Omega}&\leq \tilde{h}^N\sum_{\tilde{x}_\alpha\in \Omega\cap \tilde{h} \Z^N  }  \sum_{\tilde{x}_\beta \in \tilde{h}\Z^N} (\tilde{\psi}(\tilde{h}\alpha)-\tilde{\psi}(\tilde{h}\beta))^2 \omega(\alpha-\beta,\tilde{h})
\leq  \frac{M_\alpha C^2 L^2}{\|\psi\|_{\ell^2(\tilde{h}\Z^N)}^2}  \sum_{\tilde{x}_\alpha\in \Omega\cap \tilde{h} \Z^N  } \tilde{h}^N\\
& \leq M_\alpha C^2 L^2 \frac{|\Omega|+ o_{\tilde{h}}(1)}{\|\psi\|^2_{L^2(\Omega)}+ o_{\tilde{h}}(1)},
\end{split}
\]
where $M_\alpha=\sup_{h>0} \sum_{\alpha\not=0} \max\{|x_\alpha|^2,1\} \frac{h^N}{|x_\alpha|^{N+2s}}$. This finishes the proof for the special choice $\omega(\alpha,h)= h^N|x_\alpha|^{-N-2s}$. When $\omega(\alpha,h)\leq h^N|x_\alpha|^{-N-2s}$, let us denote the corresponding eigenvalue by $\overline{\lambda}_{h,\Omega_R}$. It is standard to check, from characterization \eqref{eq:formeigdisc}, that $\overline{\lambda}_{h,\Omega_R} \leq C\lambda_{h,\Omega_R} $ and thus, the proof follows using the previous result.
\end{proof}

In order to formulate and prove the following result we need to use the extra assumption \eqref{as:omega3} that requires the weights being radially nonincreasing.

\begin{theorem}\label{thm:rear}
Assume \eqref{as:omega1} and \eqref{as:omega3}. Let $\Omega=B_r$, the ball of radius $r>0$ centered at the origin. Then, any positive eigenfunction (solution of \eqref{eq:disceig}) satisfies that $\phi(0)=\|\phi\|_{l^\infty(h\Z^N)}$.
\end{theorem}
\begin{proof}
The proof is based on the properties of discrete Schwarz rearrangements, see \cite{P,HHH}. Note that \eqref{as:omega3} implies \eqref{as:omega2}, thus, we can take $\phi$ to be a strictly positive
eigenfunction in $\Omega$, and denote by $\phi^*$ its discrete Schwarz rearrangement. By construction,  $ \phi^*(0)=\|\phi^*\|_{l^\infty(h\Z^N)}$ and $\phi^*(x_j)=0$ for $|x_j|\ge r$. Since
$\|\phi^*\|_{l^2(h\Z^N)}=\|\phi\|_{l^2(h\Z^N)}$, and according to discrete Riesz inequality (Theorem 5.7 in \cite{HHH}) we have
$$
\begin{array}{rl}
\displaystyle
\sum_{x_\alpha\in \Omega}\sum_{x_\beta\in \Omega} \phi(x_\alpha)\phi(x_\beta) \omega (\alpha-\beta,h) \le
\displaystyle\sum_{x_\alpha\in \Omega}\sum_{x_\beta\in \Omega}  \phi^*(x_\alpha)\phi^*(x_\beta) \omega (\alpha-\beta,h) .
\end{array}
$$
Thus, thanks to the variational characterization of the first eigenvalue \eqref{eq:formeigdisc}, we get that $\phi^*$ is an eigenfunction. Finally, the simplicity of the first eigenvalue implies $\phi^*=\phi$ and the result follows.
\end{proof}

\begin{lemma}\label{lem:limrenomeig}
Let the assumptions of Theorem \ref{thm:rear}. Let also the assumptions of either Lemma \ref{lem:groeigRsecond} or  Lemma \ref{lem:decayeigfrac} hold. Consider the sequence of functions
\[
\psi_R= \phi_{h,\Omega_R}/\|\phi_{h,\Omega_R}\|_{\ell^\infty(h \Z^N)}.
\]
where $\phi_{h,\Omega_R}$ is an eigenfunction associated to $\lambda_{h,\Omega_R}$.
Then, up to a subsequence, $\psi_R \to 1$ weakly-* in $\ell^\infty(h\Z^N)$ as $R\to+\infty$.
\end{lemma}

\begin{proof}
Since $\|\psi_R\|_{\ell^\infty(h \Z^N)}=1$ for all $R>0$, there exists a subsequence $\psi_{R_k}$ that converges weakly-* in $\ell^\infty(h\Z^N)$ to some $\psi \in \ell^\infty(h \Z^N)$. In particular, this implies also pointwise convergence in $ \ell^\infty(h \Z^N)$. Since $\omega(\cdot,h)\in\ell^\infty(h\Z^N)$, we can take limits as $R\to \infty$ in \eqref{eq:disceig} and use the bounds in either Lemma \ref{lem:groeigRsecond} or  Lemma \ref{lem:decayeigfrac} to show that
$
\sum_{\beta\in\Z^N\setminus\{0\}} \left(\psi(x_\alpha+x_\beta)-\psi(x_\alpha)\right) \omega(\beta,h)=0$ for all $ x_\alpha \in h \Z^N$.
Since $\psi$ is bounded, by the Liouville theorem (see e.g. \cite{AdEJ}) we have that $\psi$ must be constant in $h \Z^N$. Moreover, Lemma \ref{thm:rear} implies that $\psi_R(0)=1$ for all $R$. By pointwise convergence, $\psi(0)=1$ and thus $\psi=1$ in $h\Z^N$.
\end{proof}

\section{Asymptotic properties of fully discrete diffusion equations}\label{comportamiento-asintotico-calor}

Fix $\tau>0$ and let $t_j=\tau j \in \tau \N$. For a function $z\in \ell^\infty(h\Z^N \times \tau \N)$ we consider the forward finite difference discrete derivative given by
\[
\partial_\tau z(x_\alpha,t_j) = \frac{z(x_\alpha,t_{j+1})-z(x_\alpha,t_{j})}{\tau} \quad \textup{for all} \quad j=0,1,\ldots
\]
In this section, we will study the asymptotic behaviour of the solution of the following explicit finite difference problem:
\begin{equation}
\label{ecuacion lineal}
\left\{
\begin{split}
\partial_{\tau} z(x_\alpha,t_j)-L_h z(x_\alpha,t_j) &= 0  \hspace{0.74cm}\quad \textup{if} \quad x_\alpha\in h \Z^N,\, t_j \in \tau \N, \\
z (x_\alpha,0) &= \varphi(x_\alpha)  \quad \textup{if} \quad x_\alpha\in h\Z^N,
\end{split}\right.
\end{equation}
for some $\varphi\in \ell^\infty(h\Z^N)$. To ensure good stability properties of the solution, we will always assume the usual CFL-type condition  given by
\begin{equation}\label{CFL2}\tag{CFL2}
\tau \leq\frac{1}{4\sum_{\alpha \in \Z^N\setminus\{0\}} \omega(\alpha,h)}.
\end{equation}

The main tool in our proofs will be the Semidiscrete Fourier Transform, see \cite{Tr,Ignat}. For any $v\in \ell^2(h\Z^N)$, it is defined by
$$
\mathcal{F}_h [v](\xi)=h^N \sum_{x_\alpha\in h\mathbb Z^N} v(x_\alpha) e^{-i\xi\cdot x_\alpha}, \qquad \xi\in Q_h:=\left[-\frac\pi{h},\frac\pi{h}\right]^N,
$$
and its inverse by
$$
v(x_\alpha) =\frac1{(2\pi)^N}\int_{Q_h} \mathcal{F}_h [v](\xi) e^{i  \xi\cdot x_\alpha} d\xi,\qquad x_\alpha\in h\mathbb Z^N.
$$

In the next two lemmas we are going to collect some properties of the Semidiscrete Fourier Transform, see \cite{Tr}. First, the classical convolution property.
\begin{lemma}\label{lem:convtoprod}
If $v\in \ell^2(h\mathbb Z^N)$ and $w\in \ell^1(h\mathbb Z^N)$, then $v*w\in \ell^2(h\mathbb Z^N)$ and
$$
\mathcal{F}_h [v*w](\xi)=\mathcal{F}_h [v](\xi) \mathcal{F}_h [w](\xi).
$$
where $*$ denotes the discrete convolution $ [v*w](x_\alpha)=h^N  \sum_{\beta\in \Z^N} v(x_\alpha-x_\beta) w(x_\beta)$.
\end{lemma}
The following result states properties of the Semidiscrete Fourier when applied to restrictions of $\R^N$ functions to a grid.

\begin{lemma} \label{lema fourier}
Let $f\in C^\infty_c(\mathbb R^N)$. For $v: h\Z^N\to \R$ defined by $v(x_\alpha)=f(x_\alpha)$ (the restriction of $f$ to the grid $h\mathbb Z^N$) we have that
\[
\sup_{\xi\in Q_h} \left| \mathcal{F}_h [v](\xi)-\widehat f(\xi)\right|=O(h^M) \quad \textup{as $h\to 0$ for all } M>0,
\]
where $\widehat f$ denotes the continuous Fourier transform of $f$. Moreover, for $\rho>0$, define  $w: h\Z^N\to \R$  by $w(x_\alpha)= f(\rho x_\alpha)$ and $\phi: \rho h \Z^N \to \R$ by $\phi(\rho h \alpha)= f(\rho h \alpha)$ for $\alpha \in \Z^N$. Then
$$
\mathcal{F}_h [w](\xi) =\rho^{-N} \mathcal{F}_{\rho h} [\phi](\xi/\rho) \quad \textup{for all} \quad \xi\in Q_h.
$$
\end{lemma}

With the property given in Lemma \ref{lem:convtoprod}, we can apply the Semidiscrete Fourier Transform to \eqref{ecuacion lineal} and get asymptotic properties. For notational simplicity, and since $h$ will be fixed in this sections, let us define $a: h\Z^N \to \R$ be defined through the weights $\omega$ by
\[
a(x_\alpha)= h^{-N} \omega(\alpha,h).
\]
Then,
\begin{equation}\label{enfourier}
\begin{split}
\partial_\tau (\mathcal{F}_h[ z])(\xi, t_j) &= \mathcal{F}_h[\partial_\tau z](\xi, t_j) = \mathcal{F}_h[ L_h z](\xi, t_j) \\
&= \textstyle h^{N} \sum_{\alpha \in \Z^N} \big(\sum_{\beta\in \Z^N} \left(z(x_\alpha+x_\beta,t_j)-z(x_\alpha,t_j)\right) h^N a(x_\beta) \big) e^{-i \xi \cdot x_\alpha}\\
&=  \mathcal{F}_h[z(\cdot,t_j)*a](\xi) - \|a\|_{\ell^1(h\Z^N)} \mathcal{F}_h[z](\xi,t_j)= (\mathcal{F}_h[a](\xi) -\mathcal{F}_h[a](0))\mathcal{F}_h[z](\xi,t_j).
\end{split}
\end{equation}
where we have used the fact that $\|a\|_{\ell^1(h\Z^N)}=\mathcal{F}_h[a](0)$. Note that $\mathcal{F}_h[a](\xi) -\mathcal{F}_h[a](0)$ is precisely the Fourier symbol of $L_h$.

As in the continuous problem $u_t=J*u-u$, the asymptotic behaviour depends on the behaviour of the symbol of the diffusion operator near the origin, see \cite{ChaCharro}. See also \cite{Ignat,Ignat-Rossi} for the semidiscrete-in-space equation
$u_t(x_\alpha,t)=a*u(x_\alpha,t) -\|a\|_{1} u(x_\alpha,t)$.
\subsection{Properties of the Fourier symbol of the difussion operator}\label{sec:symfour}
As we showed in the previous section, the Fourier symbol of the operator $L_h$ is given by
\begin{equation}\label{symbol}
m(\xi)=\mathcal F_h [a](\xi)-\mathcal F_h [a] (0).
\end{equation}
A more explicit expression of the symbol $m$ can be obtained by direct computations using the symmetry of the weights:
\begin{equation}\label{eq:weight}
m(\xi)= - h^N\sum_{\beta\in \Z^N}  a(x_\beta) (1-\cos(\xi\cdot x_\beta)).
\end{equation}
In order to prove many of the results of this paper, we will require the Fourier symbol to have one of the following properties:

\begin{equation}\label{eq:bddabove}\tag{S$_{1}$}
\exists  K>0 \quad \textup{and} \quad s\in(0,1] \quad \textup{such that} \quad m(\xi)\leq - K|\xi|^{2s}\quad \textup{for all} \quad \xi\in Q_h.
\end{equation}
\begin{equation}\label{eq:bddabovebis}\tag{S$_{2}$}
\exists K_1>0 \quad \textup{and} \quad s\in(0,1] \quad \textup{such that} \quad   -K_1|\xi|^{2s}\leq m(\xi) \quad \textup{for all} \quad \xi\in Q_h.
\end{equation}
Let us observe
\begin{equation}\label{propsimbol}
m\in L^\infty(Q_h), \mbox{ and }
m(0)=0, \, m(\xi)<0 \mbox{ otherwise}.
\end{equation}
Therefore, to obtain the previous bound it is enough to study the behaviour of the symbol near the origin.
The following result gives necessary conditions on the weight function for the symbol to satisfy one of the above properties.

\begin{lemma}\label{lem:simbolp}
Assume \eqref{as:omega1}.
\begin{enumerate}[\rm (a)]
\item\label{simbolp-item-3} Assuming the lower bound in \eqref{as:weight4},  $
\omega(\alpha,h)\geq C h^N |x_\alpha|^{-N-2s}$ for some $s\in(0,1]$, then \eqref{eq:bddabove} holds.
\item\label{simbolp-item-2} If the weight function satisfies the upper bound in \eqref{as:weight4},   $
\omega(\alpha,h)\leq C h^N |x_\alpha|^{-N-2s}$ for some $s\in(0,1)$, then \eqref{eq:bddabovebis} holds.
\item\label{simbolp-item-1} Assuming \eqref{as:weight5} holds, i.e., the weight function has a finite a second moment, then there exists  $K>0$ such that
$
\lim_{\xi\to 0} \frac{m(\xi)}{ |\xi|^{2}} =-K.
$
Therefore \eqref{eq:bddabove} and \eqref{eq:bddabovebis} hold for $s=1$.
\end{enumerate}
\end{lemma}
\begin{proof}
Let us prove \eqref{simbolp-item-1}. Given $x\in \R^N$, let us denote by $x_j$ to its $j$-th component for $j=1,\ldots,N$. Then, from \eqref{eq:weight}, we get
$$
\begin{array}{l}
\displaystyle(D m)_j (\xi)=- h^N \sum_{\beta\in\mathbb Z^N} a(x_\beta) (x_\beta)_j \sin(\xi \cdot x_\beta),\, \, \textup{and} \, \ \displaystyle(D^2 m)_{ij}(\xi)=- h^N \sum_{\beta \in \mathbb Z^N} a(x_\beta) (x_\beta)_i (x_\beta)_j \cos(\xi \cdot x_\beta).
\end{array}
$$
Clearly $(D m)_j (0)=0$. Moreover, by symmetry of the weights, $(D^2 m)_{ij}(0)=0$ when $i\ne j$ and $(D^2 m)_{ii}(0)>0$. Hence, the Hessian matrix of $m(\xi)$ at the origin is a diagonal matrix given by
$$
D^2 m(0)=-\Big( \frac{h^N}{N} \sum_{\beta \in \mathbb \Z^N} a(x_\beta) |x_\beta|^2 \Big)\cdot \mbox{Id}=
\frac{-M_2(h)}{N} \cdot \mbox{Id}.
$$
Then the Taylor expansion of $m(\xi)$ gives the desired result.

We prove now \eqref{simbolp-item-2}. Let us observe that
$$
-m(\xi)\le C \bigg(
h^N \sum_{|x_\beta||\xi|\le1} \frac{1-\cos(x_\beta\cdot\xi)}{|x_\beta|^{N+2s}} +
h^N \sum_{|x_\beta||\xi|>1} \frac{1-\cos(x_\beta\cdot\xi)}{|x_\beta|^{N+2s}}
\bigg)=S_1+S_2.
$$
Since, $1-\cos(z)\le 2$ we get
$$
S_2\le 2 h^N \sum_{|x_\beta||\xi|>1} |x_\beta|^{-N-2s} \le 2 \int_{\frac1{|\xi|}}^\infty r^{-1-2s} dr = \frac1{s} |\xi|^{2s}.
$$
On the other hand, we also have the bound $1-\cos(z)\le z^2/2$, then
$$
S_1\le \frac{|\xi|^2}{2} h^N \sum_{|x_\beta||\xi|\le 1} |x_\beta|^{2-N-2s} \le
\frac{|\xi|^2}{2} \int_0^{\frac1{|\xi|}} r^{1-2s} dr =\frac{1}{2(2-2s)} |\xi|^{2s}.
$$

Finally, we prove \eqref{simbolp-item-3}. Applying Taylor expansion, there exists $c_1>0$ such that $1-\cos(z)>c_1 z^2$ in $|z|<1$. Thus, taking $|\xi|$ small enough such that $(x_\beta\cdot\xi)\le |x_\beta||\xi|<1$ and $|x_\beta|>1/(2|\xi|)>k_0$,  we have
$$
\sum_{\beta \in  Z^N} a(x_\beta) (1-\cos(x_\beta\cdot\xi))\ge C
\sum_{\frac12<|x_\beta ||\xi|<1 } |x_\beta|^{-N-2s} (x_\beta \cdot\xi)^2.
$$
If we now restrict $x_\beta$ to the hyper-quadrant $H$ with $\xi$ as diagonal, then $\cos(x_\beta\cdot \xi)> 1/\sqrt{N}$. Define $\Omega=\{x_\alpha\in H\,:\, \frac12<|x_\alpha||\xi|<1 \}$, and note that
$$
\begin{array}{rl}
\displaystyle h^N\sum_{\frac12<|x_\beta ||\xi|<1 } |x_\beta|^{-N-2s} (x_\beta \cdot\xi)^2\ge &
\displaystyle |\xi|^2 \frac{h^N}{\sqrt N} \sum_{x_\beta \in \Omega } |x_\beta|^{2-N-2s}
\ge \displaystyle
C|\xi|^2 \int_{ \frac1{2|\xi|}}^{\frac1{|\xi|}} r^{1-2s}dr =C|\xi|^{2s}.
\end{array}
$$
The desired result follows.
\end{proof}
\subsection{Asymptotic decay (Proof of Theorem \ref{comportamiento lineal})}

To obtain an upper bound of the decay of the solution of \eqref{ecuacion lineal} we only need an upper estimate of the symbol $m$ given in \eqref{eq:bddabove}. Thanks to Lemma \ref{lem:simbolp} we obtain that in term of the behaviour of the weights at infinity.
Therefore  \ref{comportamiento lineal} is equivalent to the following Theorem.

\begin{theorem}\label{convergencia-pb-lineal}
Assume \eqref{as:omega1} and \eqref{CFL2}. Let $z$ be the solution of \eqref{ecuacion lineal} with initial datum $\varphi\in \ell^1(h\Z^N)$. If \eqref{eq:bddabove} holds, then
\begin{equation}\label{eq:asbeha}
    \limsup_{j\to\infty} \,t_j^{N/{2s}} \sup_{\alpha \in \Z^N}|z(x_\alpha,t_j)-\|\varphi\|_{\ell^1(h\Z^N)}\Gamma_s(x_\alpha,t_j)|\le C,
\end{equation}
where $\Gamma_s$ is the fundamental solution of the fractional heat equation (the heat equation if $s=1$)
$
(\Gamma_s)_t=-K (-\Delta)^{s} \Gamma_s,
$
where $K$ is the constant in \eqref{eq:bddabove}.
\end{theorem}
\begin{proof} By \eqref{enfourier} and \eqref{symbol}, we get that
\[
(\mathcal{F}_h[ z])(\xi, t_j)=(1+\tau m(\xi))^j\mathcal{F}_h[ \varphi](\xi),
\quad \textup{i.e.,} \quad z(x_\alpha,t_j)=\frac1{(2\pi)^N}\int_{Q_h}  \Big(1+\tau m(\xi) \Big)^j \mathcal{F}_h[ \varphi](\xi) e^{i x_\alpha \cdot \xi} d\xi.
\]
On the other hand, it is well known that
$$
\Gamma_{s}(x_\alpha,t_j) =\frac1{(2\pi)^N}\int_{\mathbb R^N} e^{-K|\xi|^{2s} t_j} e^{ix_\alpha\cdot\xi}d\xi.
$$
Thus, $e_{\alpha}^j=(2\pi)^N |z(x_\alpha,t_j)-\|\varphi\|_{l^1(h\Z^N)}\Gamma_s(x_\alpha,t_j)|$ satisfies
$$
e_{\alpha}^j\le  \int_{Q_h} \Big| \Big(1+\tau m(\xi)\Big)^j \mathcal{F}_h[ \varphi](\xi) -e^{-K |\xi|^{2s} t_j}  \mathcal{F}_h[ \varphi](0) \Big|d\xi  + \|\varphi\|_{\ell^1(h\Z^N)} \int_{\mathbb R^N\setminus Q_h}  e^{-K|\xi|^{2s} t_j} d\xi = I_1+I_2.
$$
To bound the second integral we perform the change of variables $\eta=\xi t_j^{1/2s}$ to get
$$
t_j^{N/2s} I_2 = \|\varphi\|_{l^1(h\Z^N)} \int_{\mathbb R^N\setminus t_j^{1/2s} Q_h}  e^{-K|\eta|^{2s}}d\eta \to 0 \quad \textup{as}\quad j\to\infty.
$$
For the first one, we have that
$$
I_1\le
\int_{Q_h} \Big| \Big(1+\tau m(\xi) \Big)^j - e^{-K|\xi|^{2s} t_j}\Big| |\mathcal{F}_h[ \varphi](\xi)|d\xi + \int_{Q_h} e^{-K|\xi|^{2s} t_j}\Big| \mathcal{F}_h[ \varphi](\xi) -\mathcal{F}_h[ \varphi](0) \Big| d\xi= I_3+I_4.
$$
As before, to estimate these integrals we perform the change $\eta=\xi t_j^{1/\alpha}$.
$$
t_j^{N/2s} I_4= \int_{t_j^{1/2s} Q_h}  e^{-K|\eta|^{2s}} \Big| \mathcal{F}_h[ \varphi]\left( \eta/{t_j^{1/2s}} \right) -\mathcal{F}_h[ \varphi](0) \Big| d\eta
\to 0 \quad \textup{as}\quad j\to\infty
$$
by the dominated convergence theorem. Finally for $I_3$ we note that from \eqref{propsimbol}, we have that $m(\xi)\leq - \delta$ if $|\xi|>\varepsilon$ and $m(\xi)\leq -K |\xi|^{2s}$ if $|\xi|\leq\varepsilon$.
Note also that \eqref{CFL2} ensures that
\[
\tau m(\xi) = - \tau \sum_{\beta\in \Z^N}  \omega(x_\beta,h) (1-\cos(\xi\cdot x_\beta)) \geq -2\tau \sum_{\beta\in \Z^N}  \omega(x_\beta,h) \geq -\frac{1}{2}.
\]
Thus,
$$
\frac12\le 1+\tau m(\xi) \le 1-\tau
\left\{
\begin{array}{ll}
 \delta, \qquad & \mbox{if } |\xi|>\varepsilon, \\
K |\xi|^{2s}, \qquad &\mbox{if } |\xi|\leq \varepsilon,
\end{array}\right.
$$
and
\begin{align*}
I_3\le & \displaystyle C \int_{Q_h} \Big| \Big(1+\tau m(\xi) \Big)^j - e^{-K|\xi|^{2s} t_j}\Big| d\xi\\
=& \displaystyle C \bigg( \int_{B_\varepsilon(0)} \Big| \Big(1+\tau m(\xi) \Big)^j - e^{-K|\xi|^{2s} t_j}\Big|  d\xi +  \int_{Q_h\setminus B_\varepsilon(0)} \Big| \Big(1+\tau m(\xi) \Big)^j - e^{-K|\xi|^{2s} t_j}\Big|d\xi\bigg)
\\
\le & \displaystyle C\bigg( \int_{B_\varepsilon(0)} \Big| \Big(1-\tau K |\xi|^{2s} \Big)^j - e^{-K|\xi|^{2s} t_j}\Big| d\xi  +  \int_{Q_h\setminus B_\varepsilon(0)} \Big| \Big(1-\tau \delta\Big)^j - e^{-K|\xi|^{2s} t_j}\Big|d\xi\bigg)
=C\left(I_5+I_6\right).
\end{align*}
To bound these integrals, we use that for $x\in(0,1)$ the function $(1-x)^j<e^{-xj}$, then
\begin{align*}
I_5\le & \displaystyle\int_{B_\varepsilon(0)} e^{-K|\xi|^{2s}t_j}\Big(1- e^{K|\xi|^{2s}t_j}\Big(1-\tau K |\xi|^{2s} \Big)^j\Big)d\xi \le
\int_{B_\varepsilon(0)} e^{-K|\xi|^{2s}t_j} d\xi= \displaystyle
t_j^{- N/2s} \int_{t^{1/2s} B_\varepsilon(0)} e^{-K|\eta|^{2s}}d\eta
\end{align*}
and
$$
I_6\le  |Q_h|  e^{-\delta t_j}+ t_j^{-N/2s} \int_{t^{1/2s} (Q_h\setminus B_\varepsilon(0))} e^{-K|\eta^2|}d\eta.
$$
Hence
$
\limsup_{t_j\to\infty} t_j^{N/2s} I_3 \le \int_{\mathbb R^N} e^{-K|\eta|^{2s}}d\eta.
$
\end{proof}
As an immediate consequence we obtain the decay of the linear problem.

\begin{corollary}\label{dacaimiento}
Under the assumption of Theorem \ref{convergencia-pb-lineal}, there exists a constant $C$ such that
$$
z(x_\alpha,t_j)\le C t_j^{-\frac{N}{2s}}.
$$
\end{corollary}
\begin{proof}
We just need to combine the decay of the fundamental solution of the associated heat equation together with the bound of Theorem \ref{convergencia-pb-lineal}.
\end{proof}

\begin{remark}\label{comportamiento}
Let us remark that if we have more information about the symbol of the operator we can take $C=0$ in the previous Theorem \ref{convergencia-pb-lineal}. Indeed, assume that
\begin{equation}\label{eq:xisq}
\lim_{\xi\to 0} m(\xi)/|\xi|^{2s}=-K.
\end{equation}
The proof remains the same as before until the estimate of $I_5$.  To do that, we define $r(t)=t^{-\gamma}$ with $\gamma<1/2s$ (notice that $r(t)\to 0$ and $t^{1/2s}r(t)\to \infty$) and split $I_5$ as follows:
$$
I_5= \int_{r(t_j)\le |\xi|<\varepsilon} \Big| \Big(1+\tau m(\xi) \Big)^j - e^{-K|\xi|^{2s} t_j}\Big|  d\xi +
\int_{|\xi|< r(t_j)} \Big| \Big(1+\tau m(\xi) \Big)^j - e^{-K|\xi|^{2s} t_j}\Big|  d\xi =I_7+I_8.
$$
To study $I_7$ we observe that  $m(\xi)\le -(K-\veps) |\xi|^{2s}$ for  $|\xi|\le \varepsilon$, then
$$
I_7\le  \int_{r(t_j)\le |\xi|<\varepsilon} \left( e^{-(K-\veps)|\xi|^{2s} t_j} +
e^{-K|\xi|^{2s} t_j}\right)d\xi=
t_j^{N/2s} \int_{t_j^{1/2s}r(t_j)\le |\xi|<t_j^{1/2s}\varepsilon} \left( e^{-(K-\veps)|\eta|^{2s} } +
e^{-K|\eta|^{2s}}\right)d\eta.
$$
As $t_j^{1/2s} r(t)\to\infty$  we get that $t^{N/2s} I_7\to 0$. To estimate $I_8$ we need a better estimate of $m$. Applying Taylor's expansion
$$
\Big(1+\tau m(\xi)\Big)^j=e^{j\log(1+\tau m(\xi))}=e^{t_j m(\xi)} e^{-\frac{t_j m^2(\xi)\tau}{2(1+ \rho(\xi))^2}}
$$
with $\rho(\xi)\in (\tau m(\xi),0)\subset (-1/2,0)$. Therefore,
$$
t_j^{N/2s} I_8=\int_{|\xi|\le t_j^{1/2s} r(t)} e^{-K|\eta|^{2s}} F(\eta,t_j) d\eta \quad \textup{with} \quad F(\eta,t_j)=\Bigg| e^{\Big(\frac{t_j m(\eta/t_j^{1/2s})}{|\eta|^{2s}} +K \Big) |\eta|^{2s}}
e^{-\frac{(t_j m(\eta/t_j^{1/2s}))^2\tau}{t_j(1+ \rho(\eta/t_j^{1/2s}))^2}}-1\Bigg|.
$$
Since $\eta/t_j^{1/2s} \to 0$  as $j\to+\infty$ we have that $ t_j m(\eta/t_j^{1/2s}) /|\eta|^{2s}+K\to 0$, which implies
$
\lim_{t_j\to\infty} F(\eta,t_j)= 0$ and $F(\eta,t_j)\le e^{\veps |\eta|^{2s}} +1$.
Thus, we can apply the dominated convergence theorem to obtain that $t^{N/2s} I_8\to0$.
\end{remark}

\section{Discrete blow-up solutions}\label{bum-discreto}

In this section, we give conditions on the solution at an arbitrary time that guarantee that the solution of problem \eqref{eq-discreta} blows up at finite time.  We show first that if a nonnegative solution is big enough for some time, then it must blow up in finite time.

\begin{lemma}\label{lem-bum}
 Assume \eqref{as:omega1} and \eqref{CFL1}. Let $u$ be a nonnegative solution of \eqref{eq-discreta}. If there exists $t_n\geq0$ such that
\begin{equation}\label{eq:bucond}
   \|u(\cdot,t_n)\|_{\ell^\infty(h\Z^N)}^{p-1} >\|\omega(\cdot,h)\|_{\ell^1(\Z^N\setminus\{0\})},
\end{equation}
 then $u$ blows up in a finite time $ T_{h,\tau}=\sum_{j=0}^\infty \tau_j<+\infty$. Moreover,
\begin{equation}\label{eq:bupratedisc}
     \lim_{j\to+\infty} (T_{h,\tau} - t_j) \|u(\cdot,t_j)\|_{\ell^\infty(h\Z^N)}^{p-1} =\frac{\tau (1+\tau)^{p-1}}{(1+\tau)^{p-1}-1}.
\end{equation}
\end{lemma}

\begin{proof} We will proceed in several steps.

\noindent\emph{Step 1.} We will prove that for all $j\geq n$ we have $\|u(\cdot,t_{j+1})\|_{\ell^\infty(h\Z^N)}>\|u(\cdot,t_{j})\|_{\ell^\infty(h\Z^N)}$ and
\begin{align*}
   \|u(\cdot,t_{j+1})\|_{\ell^\infty(h\Z^N)}& \geq  \|u(\cdot,t_{j})\|_{\ell^\infty(h\Z^N)} + \tau_j \|u(\cdot,t_{j})\|_{\ell^\infty(h\Z^N)} (\|u(\cdot,t_{j})\|_{\ell^\infty(h\Z^N)}^{p-1}-\|\omega(\cdot,h)\|_{\ell^1(\Z^N\setminus\{0\})} ).
\end{align*}
 Let us do it for $j=n$. Consider a sequence $\{\alpha_i\}_{i=1}^\infty \subset h\Z^N$ such that $
u(x_{\alpha_i},t_n) \to \|u(\cdot,t_n)\|_{\ell^\infty(h\Z^N)}$ as $i\to+\infty$.
Directly using the equation for $u$ we have
\begin{align*}
    \|u(\cdot,t_{n+1})\|_{\ell^\infty(h\Z^N)}&\geq u(x_{\alpha_i},t_{n+1})\\
    &=u(x_{\alpha_i},t_{n}) + \tau_n \textstyle\sum_{\beta\in\Z^N\setminus\{0\}} \left(u(x_{\alpha_i}+x_\beta,t_n)-u(x_{\alpha_i},t_n)\right) \omega(\beta,h) + \tau_n u^p(x_{\alpha_i},t_{n})\\
    &\textstyle\ge u(x_{\alpha_i},t_{n}) - \tau_n u(x_{\alpha_i},t_n) \sum_{\beta\in\Z^N\setminus\{0\}} \omega(\beta,h) + \tau_n u^p(x_{\alpha_i},t_{n}).
\end{align*}
Taking limits as $i\to+\infty$ in the above estimate, we get
\begin{align*}
   \|u(\cdot,t_{n+1})\|_{\ell^\infty(h\Z^N)}& \geq \textstyle  \|u(\cdot,t_{n})\|_{\ell^\infty(h\Z^N)} - \tau_n \|u(\cdot,t_{n})\|_{\ell^\infty(h\Z^N)} \sum_{\beta\in\Z^N\setminus\{0\}} \omega(\beta,h) + \tau_n \|u(\cdot,t_{n})\|_{\ell^\infty(h\Z^N)}^p\\
   &=\|u(\cdot,t_{n})\|_{\ell^\infty(h\Z^N)} + \tau_n \|u(\cdot,t_{n})\|_{\ell^\infty(h\Z^N)} (\|u(\cdot,t_{n})\|_{\ell^\infty(h\Z^N)}^{p-1}-\|\omega(\cdot,h)\|_{\ell^1(\Z^N\setminus\{0\})} )\\
   &>\|u(\cdot,t_{n})\|_{\ell^\infty(h\Z^N)}.
\end{align*}
where the last step follows from the hypothesis \eqref{eq:bucond}. Note that, in particular, we have
\[
\|u(\cdot,t_{n+1})\|_{\ell^\infty(h\Z^N)}^{p-1} >\|u(\cdot,t_n)\|_{\ell^\infty(h\Z^N)}^{p-1} >\|\omega(\cdot,h)\|_{\ell^1(\Z^N\setminus\{0\})}.
\]
Thus we can argue by induction to complete the proof of this step.

\noindent\emph{Step 2.}
Let us define the sequence
\begin{equation}\label{eq:vepsj}
\veps_j:= (\|u(\cdot,t_{j})\|_{\ell^\infty(h\Z^N)}^{p-1}-\|\omega(\cdot,h)\|_{\ell^1(\Z^N\setminus\{0\})} )/\|u(\cdot,t_{j})\|_{\ell^\infty(h\Z^N)}^{p-1} .
\end{equation}
 Note that, by Step 1, $\{\veps_{j}\}_{j=n}^\infty \subset(0,1)$ is an increasing sequence. By Step 1 and the definitions of $\veps_j$ and $\tau_j$ we have, for all $j\geq n$, that
\begin{align*}
   \|u(\cdot,t_{j+1})\|_{\ell^\infty(h\Z^N)}& \geq  \|u(\cdot,t_{j})\|_{\ell^\infty(h\Z^N)} + \tau_j \|u(\cdot,t_{j})\|_{\ell^\infty(h\Z^N)} (\|u(\cdot,t_{j})\|_{\ell^\infty(h\Z^N)}^{p-1}-\|\omega(\cdot,h)\|_{\ell^1(\Z^N\setminus\{0\})} )\\
   &= \|u(\cdot,t_{j})\|_{\ell^\infty(h\Z^N)} + \tau_j \|u(\cdot,t_{j})\|_{\ell^\infty(h\Z^N)}^p \veps_j\\
   &=\|u(\cdot,t_{j})\|_{\ell^\infty(h\Z^N)} + \tau \min\{1, \|u(\cdot,t_j)\|_{\ell^\infty(h\Z^N)}^{1-p}\} \|u(\cdot,t_{j})\|_{\ell^\infty(h\Z^N)}^p \veps_j\\
   &\le (1+\tau \veps_j)\|u(\cdot,t_{j})\|_{\ell^\infty(h\Z^N)}.
\end{align*}
In particular, this implies that $\|u(\cdot,t_j)\|_{\ell^\infty(h\Z^N)}\to+\infty$ as $j\to+\infty$, since
\begin{align*}
    \|u(\cdot,t_{j+1})\|_{\ell^\infty(h\Z^N)}&\geq  (1+\tau \veps_j)\|u(\cdot,t_{j})\|_{\ell^\infty(h\Z^N)}\geq \|u(\cdot,t_{n})\|_{\ell^\infty(h\Z^N)} \textstyle\prod_{k=n}^j (1+\tau \veps_k)\\
    &\geq \|u(\cdot,t_{n})\|_{\ell^\infty(h\Z^N)} (1+\tau \veps_n)^{j-n+1} \stackrel{j\to+\infty}{\longrightarrow} +\infty.
\end{align*}
A trivial consequence of this fact is that $\veps_j\to1$ as $j\to+\infty$ by \eqref{eq:vepsj}.

\noindent \emph{Step 3.} We will show that $u$ blows up in finite time by computing an upper bound for the blow-up time and obtain the $\leq$ part in \eqref{eq:bupratedisc}. By direct computations, for all $j\geq n$, we have that
\begin{align*}
    T_{h,\tau}-t_{j}&=\sum_{k=0}^\infty \tau_k - \sum_{k=0}^{j-1} \tau_k = \sum_{k=j}^\infty \tau_k = \tau \sum_{k=j}^\infty \|u(\cdot,t_{k})\|_{\ell^\infty(h\Z^N)}^{1-p} \leq \tau \|u(\cdot,t_{j})\|_{\ell^\infty(h\Z^N)}^{1-p}\sum_{k=j}^\infty \left(\frac{1}{(1+\tau \veps_j)^{p-1}}\right)^{k-j}\\
    &=\frac{\tau (1+\varepsilon_j\tau)^{p-1}}{(1+\varepsilon_j\tau)^{p-1}-1}\|u(\cdot,t_{j})\|_{\ell^\infty(h\Z^N)}^{1-p}.
\end{align*}
In particular, this shows that $T_{h,\tau}<+\infty$. Moreover, from the above estimate and the fact that $\veps_j\to1$ as $j\to+\infty$, we get $
\lim_{j\to+\infty} (T_{h,\tau}-t_{j})\|u(\cdot,t_{j})\|_{\ell^\infty(h\Z^N)}^{p-1} \leq  \frac{\tau (1+\tau)^{p-1}}{(1+\tau)^{p-1}-1}$.

\noindent \emph{Step 4.} Finally will obtain the $\geq$ part in \eqref{eq:bupratedisc}. Directly from the equation of $u$ we have
\[
u(x_\alpha,t_{j+1}) = u(x_\alpha,t_{j}) \Big(1-\tau_j \sum_{\beta\in\Z^N\setminus\{0\}} \omega(\beta,h)\Big) + \tau_j\sum_{\beta\in\Z^N\setminus\{0\}} u(x_\alpha+x_\beta,t_{j})\omega(\beta,h) + \tau_j u^p(x_\alpha,t_{j}).
\]
We use \eqref{CFL1} and the fact that for $j\geq n$ we have $\|u(\cdot,t_n)\|_{\ell^\infty(h\Z^N)}^{p-1} > 1$ to get
\begin{align*}
   \|u(\cdot,t_{j+1})\|_{\ell^\infty(h\Z^N)} \leq &\|u(\cdot,t_{j}) \|_{\ell^\infty(h\Z^N)}\Big(1-\tau_j \sum_{\beta\in\Z^N\setminus\{0\}} \omega(\beta,h)\Big)\\
   &+ \|u(\cdot,t_{j})\|_{\ell^\infty(h\Z^N)} \tau_j \sum_{\beta\in\Z^N\setminus\{0\}} \omega(\beta,h) +  \tau_j\|u(\cdot,t_{j})\|_{\ell^\infty(h\Z^N)}^p\\
   =&\|u(\cdot,t_{j}) \|_{\ell^\infty(h\Z^N)} + \tau_j\|u(\cdot,t_{j})\|_{\ell^\infty(h\Z^N)}^p\le (1+\tau) \|u(\cdot,t_{j}) \|_{\ell^\infty(h\Z^N)},
\end{align*}
for all $j\geq n$. We can proceed as before to get that $(T_{h,\tau}-t_{j})\|u(\cdot,t_{j})\|_{\ell^\infty(h\Z^N)}^{p-1} \geq \frac{\tau (1+\tau)^{p-1}}{(1+\tau)^{p-1}-1}$ for all $j\geq n$.
This completes the proof.
\end{proof}

\begin{proof}[Proof of Theorem \ref{tasas}] Assume that $u$ is a blow-up solution. Then, there exists a time $t_n$ such that we are in the hypothesis of Lemma \ref{lem-bum}. Therefore, there exists $t_m\ge t_n$ such that
$$
\frac12 \frac{\tau (1+\tau)^{p-1}}{(1+\tau)^{p-1}-1}
\le (T_{h,\tau}-t_{j})\|u(\cdot,t_{j})\|_{\ell^\infty(h\Z^N)}^{p-1}
\le 2 \frac{\tau (1+\tau)^{p-1}}{(1+\tau)^{p-1}-1},
$$
for all $t_j\ge t_m$. As the function $g(\tau )=\frac{\tau (1+\tau)^{p-1}}{(1+\tau)^{p-1}-1}$ is increasing we get that for $\tau\le 1$
$$
\frac{1}{p-1}=\lim_{\tau \to 0} g(\tau)\le  g(\tau)\le g(1)=
\frac{2^{p-1}}{2^{p-1}-1},
$$
and the result follows.
\end{proof}

Observe that if $L_h$ is the discretization of either the fractional Laplacian or the Laplacian  given in Remark \ref{ejemplos}, we get that $\|\omega(\cdot,h)\|_{\ell^1(\Z^N\setminus\{0\})}\to\infty$ as $h\to 0$. Thus, the above blow-up condition is not a good enough condition to obtain uniform-in-$h$ blow-up properties.  However, it suffices to give us the following property for global solutions.

\begin{corollary}\label{global-bounded}Assume \eqref{as:omega1} and \eqref{CFL1}. Let $u$ be a nonnegative solution of \eqref{eq-discreta}. If $u$ is global in time, then $u$ is bounded.
\end{corollary}

\begin{proof}
 The fact that $u$ is global in time is equivalent to say that
$\sum_{j=0}^\infty \tau_j=+\infty$.
Now assume by contradiction that $u$ is unbounded. Then, there exists $t_n$ such that $ \|u(\cdot,t_n)\|_{\ell^\infty(h\Z^N)}^{p-1} >\|\omega(\cdot,h)\|_{\ell^1(\Z^N\setminus\{0\})}$.
Thus, by Lemma \ref{lem-bum}, we have that $T^{\textup{b}}_{h,\tau}=\sum_{j=0}^\infty \tau_j<+\infty$, which is a contradiction.
\end{proof}

In order to get another blow-up property, we compare the Cauchy problem with the Dirichlet problem posed in a bounded domain $\Omega$. The proof is based on the Kapplan methods, and thus, relies on the eigenpairs of $L_h$ in $\Omega$ studied in Section \ref{sec:eigpair}.

\begin{lemma}\label{kapplan}
    Assume \eqref{as:omega1}, \eqref{as:omega2} and \eqref{CFL1} and let  $\Omega$ be a bounded domain. Let $u$ be a nonnegative solution of \eqref{eq-discreta} and $\lambda_{h,\Omega}$ be the first eigenvalue of $L_h$ in $\Omega$ with an associated eigenfunction $\phi$ (i.e $\phi$ satisfies \eqref{eq:disceig}) such that $\|\phi\|_{\ell^1(h\Z^N)}=1$. If
$
h^N\sum_{x_\alpha\in\Omega} u(x_\alpha,0) \phi(x_\alpha) > (\lambda_{h,\Omega})^{\frac{1}{p-1}},
$
then $u$ blows up in finite time.
\end{lemma}
\begin{proof}
Let us consider the following discrete Dirichlet problem
\begin{equation}\label{Dirichlet-discreto}
\left\{
\begin{array}{ll}
\partial_{\tau_j} v(x_\alpha,t_j)=L_h v(x_\alpha,t_j)+v^p(x_\alpha,t_j),\quad & x_\alpha\in \Omega,\, t_j\in(0,T), \\
v(x_\alpha,t_j)=0, & x_\alpha\in \R^N\setminus\Omega,\, t_j\in(0,T),\\
v(x_\alpha,0)= u(x_\alpha,0), & x_\alpha\in\Omega,
\end{array}\right.
\end{equation}
Note that $\tau_j$ depends on the solution $u$ of the Cauchy problem \eqref{eq-discreta} through \eqref{CFL1}. Thus, it is standard to check that $v\leq u$. As a consequence, we only need to show that $v$ blows up in finite time. Let us define
$
I(t_j)\coloneqq h^N \sum_{x_\alpha\in \Omega} v(x_\alpha,t_j)\phi(x_\alpha).
$
Note that, by Proposition \ref{prop:eig}, we have that $\phi\geq0$ and $\phi(x_\alpha)>0$ if $x_\alpha\in\Omega$.
Multiplying the equation of Problem \eqref{Dirichlet-discreto} by $h^N\phi(x_\alpha)$ and performing the sum over $x_\alpha\in \Omega$ we obtain that
$$
\partial_{\tau_j} I(t_j) =h^N\sum_{x_\alpha\in\Omega} L_h v(x_\alpha,t_j) \phi(x_\alpha)+ h^N\sum_{x_\alpha\in\Omega} v^p(x_\alpha,t_j) \phi(x_\alpha).
$$
On the one hand, we can apply Jensen's inequality to get
$$
h^N\sum_{x_\alpha\in\Omega} v^p(x_\alpha,t_j) \phi(x_\alpha)  \ge \Big(h^N\sum_{x_\alpha \in\Omega} v(x_\alpha,t_j) \phi(x_\alpha)\Big)^p=I(t_j)^p.
$$
On the other hand, by \eqref{as:omega1} we have that $L_h$ is self-adjoint, then
\[
h^N\sum_{x_\alpha\in\Omega} L_h v(x_\alpha,t_j) \phi(x_\alpha) = h^N\sum_{x_\alpha\in\Omega} v(x_\alpha,t_j) L_h \phi(x_\alpha) = -\lambda_{h,\Omega}h^N \sum_{x_\alpha\in\Omega} v(x_\alpha,t_j)\phi(x_\alpha)=- \lambda_{h,\Omega} I(t_j).
\]
Therefore,
\begin{equation}\label{kaplan}
\partial_{\tau_j} I(t_j) \ge I(t_j)^p -\lambda_{h,\Omega}  I(t_j)= I(t_j) (I(t_j)^{p-1}-\lambda_{h,\Omega} ).
\end{equation}
Since, by hypothesis, $I(0)^{p-1}>\lambda_{h,\Omega}>0$, we have that $\{I(t_j)\}_{j=1}^\infty$ is an increasing sequence.

Let us argue by contraction, and assume that $u$ is a global-in-time solution. By Corollary \ref{global-bounded}, we then have that $u$ is bounded (and let us write $u\leq K$ for some $K>0$). This implies that $\tau_j$ is bounded from below since $
\tau_j=\tau \min\{1, \|u(\cdot,t_j)\|_{\ell^\infty(h\Z^N)}^{1-p}\} \geq \tau \min\{1, K^{1-p}\}$,
and $I(t_j)$ is bounded from above since
\[
I(t_j) \leq \|v(\cdot,t_j)\|_{\ell^\infty(h\Z^N)} \|\phi\|_{\ell^1(h\Z^N)} \leq \|u(\cdot,t_j)\|_{\ell^\infty(h\Z^N)}\leq K.
\]
Therefore, $I(t_j)$ is an increasing sequence bounded from above, which in particular shows that there exists $K_0$ such that
$
\lim_{j\to+\infty} I(t_j)=K_0>I(0).
$
We also have that
\[
0\leq \lim_{j\to+\infty} \partial_{\tau_j} I(t_j) = \lim_{j\to+\infty} \frac{I(t_{j+1})-I(t_j)}{\tau_j} \leq \lim_{j\to+\infty} \frac{I(t_{j+1})-I(t_j)}{\tau \min\{1, K^{1-p}\}} =\frac{K_0-K_0}{\tau \min\{1, K^{1-p}\}}=0.
\]
From here, we reach a contradiction taking limits in \eqref{kaplan}:
\[
0=\lim_{j\to+\infty} \partial_{\tau_j} I(t_j) \geq \lim_{j\to+\infty} \left(I(t_j) (I(t_j)^{p-1}-\lambda_{h,\Omega} )\right) = K_0 (K_0^{p-1}-\lambda_{h,\Omega} ) > K_0 (I(0)^{p-1}-\lambda_{h,\Omega} )>0. \qedhere
\]
\end{proof}

\section{The Fujita exponent for the discrete equation}\label{sec-fujita}

The main goal of this section is to prove Theorem \ref{teo.fujita1}. First we show that, for $s\in(0,1]$, there exist global solutions of \eqref{eq-discreta} in the range $p>1+\frac{2s}{N}$ under the condition that the diffusion operator $L_h$ has a symbol $m$ such that $m(\xi)\leq -K |\xi|^{2s}$ for $|\xi|<\xi_0$ (see Lemma \ref{lem:simbolp} for sufficient conditions on $L_h$ for this property to hold).

\begin{lemma}[Global solutions for $p>1+\frac{2s}{N}$]  \label{hay globales} Assume \eqref{as:omega1} and \eqref{CFL1}.
Let $L_h$ have a Fourier symbol $m(\xi)$ such that either \eqref{eq:bddabove} or \eqref{eq:xisq} hold
for some $K>0$ and $s\in(0,1]$. Then, for $p>1+\frac{2s}{N}$, there exists nonnegative nontrivial global solutions of \eqref{eq-discreta}.

\end{lemma}
\begin{proof}
    The idea of the proof is to construct a bounded supersolution to the solution $u$ of problem \eqref{eq-discreta}. First, we note that if we can show that $u\leq 1$, we have, from \eqref{CFL1}, that $\tau_j=\tau$.

    To do so, we fix $\tau>0$ as usual, define a fixed time step grid $t_j=\tau j$, and consider $z$ to be the solution of the linear problem
\begin{equation*}
\left\{
\begin{split}
\partial_{\tau} z(x_\alpha,t_j)-L_h z(x_\alpha,t_j) &= 0,  \hspace{0.74cm}\quad \textup{if} \quad x_\alpha\in h \Z^N,\, t_j \in \tau \N, \\
z (x_\alpha,0) &= \varphi(x_\alpha),  \quad \textup{if} \quad x_\alpha\in h\Z^N,
\end{split}\right.
\end{equation*}
for some $\varphi\in \ell^1(h\Z^N)$. Given $J\in \N$ and $\lambda\in(0,1)$ let us also define the following function that will be our candidate for supersolution:
$$
\overline u(x_\alpha,t_j)\coloneqq \mathcal{A}(t_j)  z(x_\alpha,t_J+t_j) \quad \textup{with} \quad \mathcal{A}(t_j)=(t_J+t_j)^\lambda.
$$
Let us show first that we have $\overline{u}\leq 1$ for an appropriate choice of the parameters $J$ and $\lambda$. Indeed, by Theorem \ref{convergencia-pb-lineal} (asymptotic behaviour of the linear equation), we have that there exists a constant $C>0$
\[
\overline u(x_\alpha,t_j)\leq (t_J+t_j)^\lambda  z(x_\alpha,t_J+t_j) \leq C (t_J+t_j)^{\lambda-\frac{N}{2s}}.
\]
So it is enough to take $\lambda < \min\{\frac{N}{2s},1\}$ and $J$ big enough to get that $\overline{u}\leq 1$. At this point, if we show that $\overline{u}$ satisfies $\partial_{\tau_j} \overline{u}-L_h \overline{u}-\overline{u}^p\geq0$,
this suffices to show that any solution $u$ of \eqref{eq-discreta} with $u(x_\alpha,0)\leq \overline{u}(x_\alpha,0)$ will be global in time. This is a consequence of the fact that, at every time step $u(x_\alpha,t_j)\leq \overline{u}(x_\alpha,t_j) \leq 1$ and thus $\tau_j=\tau$.
To prove this, let us note that
\[
L_h \overline{u}(x_\alpha,t_j)=\mathcal{A}(t_j) L_h z(x_\alpha,t_J+t_j)=\mathcal{A}(t_j) \partial_{\tau} z(x_\alpha,t_J+t_j)
\]
and
\[
\partial_\tau \overline u(x_\alpha,t_j)= \partial_\tau \mathcal{A}(t_j) z(x_\alpha,t_J+t_{j+1})+  \mathcal{A}(t_j) \partial_{\tau} z(x_\alpha,t_J+t_j).
\]
Thus,
\begin{equation}\label{eq.super}
\partial_\tau \overline u(x_\alpha,t_j)-L_h \overline{u}(x_\alpha,t_j)- \overline u^p(x_\alpha,t_j)= \partial_\tau \mathcal{A}(t_j) z(x_\alpha,t_J+t_{j+1}) - \mathcal{A}(t_j)^p z^p(x_\alpha,t_J+t_j).
\end{equation}
Let us estimate the terms of the right hand side of \eqref{eq.super}. By definition of $z$, we have
\begin{align*}
z(x_\alpha,t_J+t_{j+1})&= z(x_\alpha,t_J+t_{j}) \bigg(1-\tau \sum_{\beta\in\Z^N\setminus\{0\}} \omega(\beta,h)\bigg) + \tau \sum_{\beta\in\Z^N\setminus\{0\}} z(x_\alpha+x_\beta,t_J+t_{j}) \omega(\beta,h) \\
&\geq z(x_\alpha,t_J+t_{j}) \bigg(1-\tau \sum_{\beta\in\Z^N\setminus\{0\}} \omega(\beta,h)\bigg)\geq \frac{1}{2}z(x_\alpha,t_J+t_{j}),
\end{align*}
where the last inequality follows from \eqref{CFL1}.
We also have that, for some $\xi\in[t_J+t_j,t_J+t_{j+1}]$, the following estimate holds:
\begin{align*}
  \partial_\tau \mathcal{A}(t_j)  = \frac{\mathcal{A}(t_{j+1})-\mathcal{A}(t_j)}{\tau}= \lambda \xi^{\lambda-1} \geq \lambda (t_J+t_{j+1})^{\lambda-1} =\lambda (t_J+t_{j}+\tau)^{\lambda-1} \geq \frac{\lambda}{2} (t_J+t_{j})^{\lambda-1},
\end{align*}
where the last inequality follows from the fact that $J\geq1$. Finally, by Theorem \ref{convergencia-pb-lineal}, we also have
\begin{align*}
    \mathcal{A}(t_j)^p z^p(x_\alpha,t_J+t_j) = (t_J+t_j)^{\lambda p} z^p(x_\alpha,t_J+t_{j}) \leq C^{p-1} (t_J+t_j)^{p\lambda-(p-1)\frac{N} {2s}}z(x_\alpha,t_J+t_j).
\end{align*}
Let us denote $\beta\coloneqq p\lambda-(p-1)\frac{N} {2s}$. Thus, from \eqref{eq.super}, we get
\begin{align*}
  \partial_\tau \overline u(x_\alpha,t_j)-L_h \overline{u}(x_\alpha,t_j)- \overline u^p(x_\alpha,t_j) &\geq z(x_\alpha,t_J+t_j) \Big( \frac{\lambda}{4} (t_J+t_{j})^{\lambda-1} -C^{p-1} (t_J+t_j)^{\beta} \Big)  \\
  &= z(x_\alpha,t_J+t_j) (t_J+t_j)^{\beta} \Big( \frac{\lambda}{4} (t_J+t_{j})^{\lambda-1-\beta} -C^{p-1} \Big).
\end{align*}
Here we just need to note that $\lambda-1-\beta>0$ if $\lambda <\frac{N}{2s}-\frac{1}{p-1}$. Note that we can always choose $\lambda>0$ satisfying this since we are in the range $p>1+\frac{2s}{N}$, where $\frac{N}{2s}-\frac{1}{p-1}>0$. Thus, we choose $\lambda<\min\{1,\frac{N}{2s}-\frac{1}{p-1}\}$ to ensure that $\lambda \in(0,1)$. Finally, we choose $J$ big enough such that
$
(t_J)^{\lambda-1-\beta} \geq 4C^{p-1}/\lambda,
$
and conclude that
\[
  \partial_\tau \overline u(x_\alpha,t_j)-L_h \overline{u}(x_\alpha,t_j)- \overline u^p(x_\alpha,t_j) \quad \textup{for all} \quad t_j\in [0,\infty), \, x_\alpha \in h\Z^N,
\]
which is the end of the proof.
\end{proof}

We will show now that in the range $p<1+\frac{2s}{N}$, all nonnegative nontrivial solutions blow up in finite time under some condition on the scaling of the first eigenvalue $\lambda_{h,\Omega}$ of $L_h$ as the domain $\Omega$ widens. More precisely, we will require $\lambda_{h,\Omega_R} \leq C R^{-2s}$ for some $s\in(0,1]$ where $\Omega_R=R\Omega$ for $R>0$ (see Lemma \ref{lem:groeigRsecond} and Lemma \ref{lem:decayeigfrac}  for sufficient conditions on $\omega(\alpha,h)$ to satisfy this assumption).

\begin{lemma}[Solutions blow up for $p<1+\frac{2s}{N}$] \label{todobum}
Assume \eqref{as:omega1}, \eqref{as:omega2}, \eqref{as:omega3} and \eqref{CFL1} and let $\Omega=B_1(0)$.
Let $\lambda_{h,\Omega}$ be the first eigenvalue of $L_h$ in $\Omega$.
Assume also that $\lambda_{h,\Omega_R} \leq C R^{-2s}$ for some $s\in(0,1]$.  Then, for $p<1+\frac{2s}{N}$, any nontrivial nonnegative solution of \eqref{eq-discreta} blows up in finite time.
\end{lemma}

\begin{proof}
Without loss of generality we can assume that $u(0,0)>0$. This is due to the fact that he solution is nontrivial and nonnegative and the equation is translation-invariant.

Let $\phi_{h,\Omega_R}$ an eigenfunction of $L_h$ in $\Omega_R$ associated to the first eigenvalue $\lambda_{h,\Omega_R}$ and such that $\|\phi_{h,\Omega_R}\|_{\ell^1(h\Z^N)}=1$. Let us recall now that, by Lemma \ref{kapplan}, if
\begin{equation}\label{eq.kapplan}
\sum_{x_\alpha\in\Omega_R} u (x_\alpha,0) \phi_{h,\Omega_R}(x_\alpha) > (\lambda_{h,\Omega_R})^{\frac{1}{p-1}},
\end{equation}
then $u$ blows up. Let us observe that
\begin{align*}
1&= \|\phi_{h,\Omega_R}\|_{\ell^\infty(h\Z^N)} h^N\sum_{x_\alpha\in\Omega_R}  \frac{\phi_{h,\Omega_R}(x_\alpha)}{\|\phi_{h,\Omega_R}\|_{\ell^\infty(h\Z^N)}} \leq \|\phi_{h,\Omega_R}\|_{\ell^\infty(h\Z^N)} h^N\sum_{x_\alpha\in\Omega_R} 1 \leq \tilde{C} \|\phi_{h,\Omega_R}\|_{\ell^\infty(h\Z^N)}R^{N},
\end{align*}
that is,
\begin{align}\label{eq:linfboundauto}\|\phi_{h,\Omega_R}\|_{\ell^\infty(h\Z^N)}\geq R^{-N}/\tilde{C}.
\end{align}
Now recall that  we are under the hypothesis of Theorem \ref{thm:rear} the eigenfunction reaches its maximum at the origin, $\|\phi_{h,\Omega_R}\|_{\ell^\infty(h\Z^N)}=\phi_{h,\Omega_R}(0)$.
Then,
\begin{align*}
\textstyle \sum_{x_\alpha\in\Omega_R} u (x_\alpha,0) \phi_h(x_\alpha) &> u(0,0) \phi_h(0)=u(0,0) \|\phi_h\|_{\ell^\infty(h\Z^N)} \geq  u(0,0)R^{-N}/\tilde{C}.
\end{align*}
And on the other hand, by hypothesis, $
 (\lambda_{h,\Omega_R})^{\frac{1}{p-1}} \leq C^{\frac{1}{p-1}} R^{-\frac{2s}{p-1}}$.
Thus, \eqref{eq.kapplan} is satisfied if
$
u(0,0)R^{-N} \geq \tilde{C}C^{\frac{1}{p-1}} R^{-\frac{2s}{p-1}},
$
which holds for $R$ big enough if $N<\frac{2s}{p-1}$, that is, $p<1+\frac{2s}{N}$.
\end{proof}

To formulate the following result, we need to consider the grid in time in a more precise way. A solution $u(x_\alpha,t_j)$ of \eqref{eq-discreta} is defined for all
\[
\textstyle x_\alpha \in h\Z^N, \quad t_0=0 \quad \textup{and} \quad t_j=\sum_{k=0}^{j-1} \tau_k \quad \textup{with} \quad \tau_j:= \tau \min\{1, \|u(\cdot,t_j)\|_{\ell^\infty(h\Z^N)}^{1-p}\}.
\]
Note that the time grid is generated together with the solution. For notational convenience, given a solution $u$ of \eqref{eq-discreta}, consider the time grid  $\mathcal{T}_{\tau}=\bigcup_{j=0}^\infty \{t_j\}$, so that $u:h\Z^N\times \mathcal{T}_\tau \to \R_+$.
Finally, we define the $\ell^p$ norm associated to such functions by
\[
\|u\|_{\ell^p(h\Z^N\times \mathcal{T}_\tau)}^p= h^{N} \sum_{j=0}^{\infty} \tau_j \sum_{x_\alpha \in h\Z^N}u(t_j,x_\alpha)^p.
\]

\begin{remark}
    Note that contrary to the $\ell^p$ spaces in regular grids, where the inclusion $\ell^p\subset \ell^\infty$ holds, it
    might a priori not be true in this case. This is due to the fact that, if $u(x_\alpha,t_j)\to+\infty$ as $j\to+\infty$, then $\tau_j\to0$ as $j\to+\infty$, and there could be compensations that still make the $\ell^p$-norm finite.
\end{remark}

\begin{lemma}\label{no globales}
Assume \eqref{as:omega1}, \eqref{CFL1} and $p=1+\frac{2s}{N}$.  Let $L_h$ have a Fourier symbol $m(\xi)$ such that \eqref{eq:bddabovebis} holds. If $u\in \ell^p(h\Z^N\times \mathcal{T}_\tau)$ is a nonnegative global solution of \eqref{eq-discreta}, then $u$ is identically zero.
\end{lemma}

\begin{proof}
Let us argue by contradiction and assume that there exists a nonnegative nontrivial global solution $u$ of \eqref{eq-discreta} such that $u\in \ell^p(h\Z^N\times \mathcal{T}_\tau)$ . The fact that $u$ is global implies, by Corollary \ref{global-bounded}, that $u$ is bounded. Let us write $
B\coloneqq\sup_{j\in \N}\|u(\cdot,t_j)\|_{\ell^\infty(h\Z^N)}$
and note that, by \eqref{CFL1}, we have that
\begin{equation}\label{cota tau}
\min\{1,B^{1-p}\}\tau \leq \tau_j\leq \tau.
\end{equation}
Let us also consider a nonincreasing cut-off function $\rho\in C_{\textup{c}}^\infty([0,\infty))$  such that $\rho(\eta)=1$ on $\eta \in[0,1]$ and $\rho(\eta)=0$ for $\eta>2$. Let also $\overline{\rho}(x)=\rho(|x|)$ for $x\in \R^N$. For $R>0$, $\veps>0$ and $J\in\N$ given we define

$$
\psi(t) =\rho\left(\frac{t-t_{J}}{R^{2s}}\right)\quad \textup{for} \quad t\geq t_J \quad \textup{and} \quad \psi(t)=0 \quad \textup{for} \quad t< t_J,
$$
and
$
\phi(x)=\overline{\rho}\left(\varepsilon \frac{x}{R}\right)$ for  $x\in \R^d$.
Multiplying equation \eqref{eq-discreta} by $h^N\tau_j\psi(t_j)\phi(x_\alpha)$ and summing over $ h\mathbb Z^N \times \{j\ge J\}$ we get
\begin{align*}
h^N\sum_{j= J}^{\infty}& \tau_j \sum_{x_\alpha\in h\mathbb Z^N} u^{p}(x_\alpha,t_j)\psi(t_j)\phi(x_\alpha) \\&=
h^N\sum_{j= J}^{\infty} \tau_j \sum_{x_\alpha\in h\mathbb Z^N} \partial_{\tau_j} u(x_\alpha,t_j) \psi(t_j)\phi(x_\alpha)
-h^N\sum_{j= J}^\infty\tau_j  \sum_{x_\alpha\in h\mathbb Z^N} L_h u(x_\alpha,t_j) \psi(t_j)\phi(x_\alpha)
=:  S_1+S_2.
\end{align*}
Note that the above sums are finite since $u$ is bounded and $\psi\phi$ is compactly supported in $\R^N\times[0,\infty)$.
Let us bound $S_1$ first. To do so, we sum by parts in the time variable to get

\begin{align*}
    S_1=&\displaystyle - h^N\sum_{x_\alpha\in h\mathbb Z^N}
 u(x_\alpha,t_{J}) \psi(t_{J})\phi(x_\alpha) - h^N \sum_{j= J}^\infty \sum_{x_\alpha\in h\mathbb Z^N} u(x_\alpha,t_{j+1})\phi(x_\alpha) \left(\psi(t_{j+1}) -\psi(t_{j})\right)\\
\leq & - h^N\sum_{j= J}^\infty \tau_j \sum_{x_\alpha\in h\mathbb Z^N} u(x_\alpha,t_{j+1})\phi(x_\alpha) \partial_{\tau_j} \psi(t_j).
\end{align*}
We observe now that $\partial_{\tau_j} \psi(t_j)=0$ if either $t_j>t_{J}+2R^{2s}$ or $t_{j+1}\le t_{J}+R^{2s}$ (i.e. $t_{j}\le t_{J}+R^{2s}-\tau_j$). Thus, \eqref{cota tau} implies that the  support of $\partial_{\tau_j}\psi$ is included in $I_R=(t_{J}+R^{2s}- \tau, t_{J}+2 R^{2s})$. Moreover, since $\psi$ is nonincreasing for $t\geq t_J$, using the Taylor expansion
$$
-\partial_{\tau_j} \psi(t_j)= |\partial_{\tau_j} \psi(t_j)|\le \frac{\|\rho\|_{C^1(\R)}}{R^{2s}}\mathds{1}_{I_R}(t_j).
$$
for all $t_j\geq t_J$. Since $\phi(x_\alpha)\leq  \mathds{1}_{B_{\frac{2R}{\varepsilon}}(0)}(x_\alpha)$, we can use H\"older's inequality to get
\begin{align*}
    S_1\leq &\frac{\|\rho\|_{C^1(\R)}}{R^{2s}} h^N \sum_{t_j \in I_R} \tau_j \sum_{|x_\alpha| \leq \frac{2R}{\veps}} u(x_\alpha,t_{j+1})\\
    &\leq \frac{\|\rho\|_{C^1(\R)}}{R^{2s}} \bigg( h^N \sum_{t_j \in I_R} \tau_j \sum_{|x_\alpha| \leq \frac{2R}{\veps}} u^p(x_\alpha,t_{j+1})\bigg)^{\frac{1}{p}}\bigg( h^N \sum_{t_j \in I_R} \tau_j \sum_{|x_\alpha| \leq \frac{2R}{\veps}} 1\bigg)^{\frac{p-1}{p}}\\
    &\leq C  \frac{\|\rho\|_{C^1(\R)}}{R^{2s}} |I_R|^{\frac{p-1}{p}} |B_{\frac{2R}{\veps}}|^{\frac{p-1}{p}} \bigg( h^N \sum_{t_j \in I_R} \tau_j \sum_{|x_\alpha| \leq \frac{2R}{\veps}} u^p(x_\alpha,t_{j+1})\bigg)^{\frac{1}{p}}
    \\
    &\leq \tilde{C} \veps^{-N\frac{p-1}{p}} \bigg( h^N \sum_{t_j \in I_R} \tau_j \sum_{x_\alpha \in h\Z^N} u^p(x_\alpha,t_{j+1})\bigg)^{\frac{1}{p}}.
\end{align*}
Finally, we note that $\mathds{1}_{I_R}$ converges to zero pointwise as $R\to+\infty$. Thus, since $u\in \ell^p(h\Z^N\times \mathcal{T}_\tau)$, we have that $\limsup_{R\to+\infty}S_1 \leq 0$.

We bound $S_2$ by using the symmetry of $L_h$ to get $
S_2=h^N\sum_{j= J}^\infty\tau_j  \sum_{x_\alpha\in h\mathbb Z^N} u(x_\alpha,t_j) \psi(t_j)(-L_h \phi(x_\alpha))$.
Observe that if  $|x_\alpha|>2R/\veps$ then $\phi(x_\alpha)=0$, and thus  $-L_h \phi(x_\alpha)\leq0$. Then,
\begin{align*}
S_2&\leq h^N\sum_{j= J}^\infty\tau_j  \sum_{|x_\alpha|<\frac{2R}{\veps}} u(x_\alpha,t_j) \psi(t_j)(-L_h \phi(x_\alpha))\leq h^N\sum_{j= J}^\infty\tau_j  \sum_{|x_\alpha|<\frac{2R}{\veps}} u(x_\alpha,t_j) \psi(t_j)|L_h \phi(x_\alpha)| .
\end{align*}
To bound $L_h\phi(x_\alpha)$ for $|x_\alpha|<2R/\varepsilon$ we use  the Fourier variables. More precisely, by scaling of the discrete Fourier transform (Lemma \ref{lema fourier}), we get
\begin{align*}
|L_h\phi(x_\alpha)| =&\left| \frac{1}{(2\pi)^N} \int_{Q_h} m(\xi) \mathcal F_h [\phi](\xi) e^{i\xi\cdot x_\alpha}d\xi \right| \leq  \Big(\frac{R}{2\pi\varepsilon}\Big)^N
\displaystyle\int_{Q_h} |m(\xi)| \left|\mathcal F_{\frac{\varepsilon }{R}h}[\overline\rho](\frac{R}{\varepsilon} \xi)\right| d\xi \\
= & \displaystyle \frac{1}{(2\pi)^N}
\int_{Q_{\frac{R}{\veps}h}} \left|m(\frac{\varepsilon}{R}\eta)\right| \left|\mathcal F_{\frac{\varepsilon h}{R}} [\overline\rho](\eta) \right|d\eta \leq \displaystyle K_1 \left(\frac{\varepsilon}{R}\right)^{2s}
\int_{Q_{\frac{R}{\veps}h}} |\eta|^{2s} \left|\mathcal F_{\frac{\varepsilon h}{R}} [\overline\rho](\eta) \right|d\eta=K_1 \Big(\frac{\varepsilon}{R}\Big)^{2s} I_1,
\end{align*}
where the last inequality follows from the estimate on the symbol given by \eqref{eq:bddabovebis}.
Now, using again Lemma \ref{lema fourier}, together with the fact that $\overline{\rho}\in C_\textup{c}^\infty(\mathbb{R}^N)$,
\begin{align*}
    I_1&\leq \int_{Q_{\frac{R}{\veps}h}} |\eta|^{2s} |\mathcal F_{\frac{\varepsilon h}{R}} [\overline\rho](\eta)-\widehat{\overline\rho}(\eta)|d\eta + \int_{\mathbb R^N} |\eta|^{2s} |\widehat{\overline\rho}(\eta)|d \eta \\
    &\leq O\Big(\Big(\frac{\veps h}{R}\Big)^M\Big) \int_{Q_{\frac{R}{\veps}h}} |\eta|^{2s}d \eta + \int_{\mathbb R^N} |\eta|^{2s} |\widehat{\overline\rho}(\eta)|d \eta \leq O\Big(\Big(\frac{\veps h}{R}\Big)^{M-N-2s}\Big) + \int_{\mathbb R^N} |\eta|^{2s} |\widehat{\overline\rho}(\eta)|d \eta,
\end{align*}
with $M>N+2s$. Then, for $R$ large enough, we get that $I_1\leq C$ for some positive constant $C$.  Thus,
$
|L_h\phi(x_\alpha)|\le \widetilde{C} \left(\varepsilon/R\right)^{2s}.
$
This implies that
\begin{align*}
S_2&\le C \left(\frac{\varepsilon}{R}\right)^{2s}
h^N\sum_{t_j\in  I_R}  \tau_j\sum_{|x_\alpha|<\frac{2R}{\veps}}  u(x_k,t_j)\leq C \left(\frac{\varepsilon}{R}\right)^{2s} |I_R|^{\frac{p-1}{p}} |B_{\frac{2R}{\veps}}|^{\frac{p-1}{p}} \bigg( h^N \sum_{j=J}^\infty \tau_j \sum_{x_\alpha\in h\Z^N } u^p(x_\alpha,t_{j+1})\bigg)^{\frac{1}{p}}\\
&\leq C \veps^{N\frac{(p-1)^2}{p}}\bigg( h^N \sum_{j=J}^\infty \tau_j \sum_{x_\alpha\in h\Z^N } u^p(x_\alpha,t_{j+1})\bigg)^{\frac{1}{p}}.
\end{align*}
Summarizing, we have obtained
\[
h^N\sum_{j= J}^{\infty} \tau_j \sum_{x_\alpha\in h\mathbb Z^N} u^{p}(x_\alpha,t_j)\psi(t_j)\phi(x_\alpha) \leq S_1 + C \veps^{N\frac{(p-1)^2}{p}}\bigg( h^N \sum_{j=J}^\infty \tau_j \sum_{x_\alpha\in h\Z^N } u^p(x_\alpha,t_{j+1})\bigg)^{\frac{1}{p}}.
\]
Taking limits as $R\to+\infty$ leads to
\[
\bigg(h^N\sum_{j= J}^{\infty} \tau_j \sum_{x_\alpha\in h\mathbb Z^N} u^{p}(x_\alpha,t_j) \bigg)^{\frac{p-1}{p}}\leq  C \veps^{N\frac{(p-1)^2}{p}}.
\]
Thanks to the arbitrariness of $J$ and $\veps$, we conclude that $u$ must be zero, which is a contradiction with the fact that $u$ is a nontrivial solution.
\end{proof}

\begin{theorem}\label{fujitabum}
Assume \eqref{as:omega1}, \eqref{as:omega3}, \eqref{as:omega2}, \eqref{CFL1} and $p=1+\frac{2s}{N}$. Let $\lambda_{h,B_R}$ be the first eigenvalue of $L_h$ in $B_R$ for $R>0$ and an associated eigenfunction $\phi_{h,B_R}$ (i.e $\phi$ satisfies \eqref{eq:disceig}) such that $\|\phi_{h,B_R}\|_{\ell^1(h\Z^N)}=1$. Let the following assumptions hold:
\begin{enumerate}[\rm (a)]
\item $\lambda_{h,B_R} \leq C R^{-2s}$ for some $s\in(0,1]$.
\item$L_h$ has a Fourier symbol $m(\xi)$ such that \eqref{eq:bddabovebis} holds.
\end{enumerate} Then every nonnegative nontrivial solution of \eqref{eq-discreta}  blows up in finite time.
\end{theorem}

\begin{proof}
Assume, by contradiction, that there exists a global solution $u$. Then, by Lemma \ref{kapplan}, for any $j\in \N$, we have that
\[
h^N\sum_{x_\alpha\in B_R} u(x_\alpha,t_j) \phi_{h,B_R}(x_\alpha) \leq (\lambda_{h,B_R})^{\frac{1}{p-1}},
\]
with $\lambda_{h,B_R}$ being the first eigenvalue of $L_h$ in $B_R$ and $\phi_{h,B_R}$ an associated eigenfunction (i.e. satisfying \eqref{eq:disceig}) and such that $\|\phi_{h,B_R}\|_{\ell^1(h\Z^N)}=1$.
Thus,
\[
h^N\sum_{x_\alpha\in B_R} u(x_\alpha,t_j) \frac{\phi_{h,B_R}(x_\alpha)}{\|\phi_{h,B_R}\|_{\ell^\infty(h\Z^N)}} \leq \frac{(\lambda_{h,B_R})^{\frac{1}{p-1}}}{\|\phi_{h,B_R}\|_{\ell^\infty(h\Z^N)}}\leq C \frac{R^N}{R^{\frac{2s}{p-1}}}=C,
\]
where we have used \eqref{eq:linfboundauto} to bound $\|\phi_{h,B_R}\|_{\ell^\infty(h\Z^N)}^{-1}$, and also used the decay of the eigenvalue. We take now limit as $R\to+\infty$ and use Lemma \ref{lem:limrenomeig}, which ensures that $\phi_{h,B_R}/\|\phi_{h,\Omega_R}\|_{\ell^\infty(h\Z^N)} \to 1$ pointwise (up to a subsequence), to get $
h^N\sum_{x_\alpha\in h \Z^N} u(x_\alpha,t_j) \leq C$,
that is, we have that $u(\cdot,t_j)\in\ell^1(h\Z^N)$ uniformly in time. We can the sum in space \eqref{eq-discreta} to obtain
$\|u(\cdot,t_{j+1})\|_{\ell^1(h\Z^N)} - \|u(\cdot,t_{j})\|_{\ell^1(h\Z^N)} = \tau_j\|u(\cdot,t_{j})\|_{\ell^p(h\Z^N)}^p$.
Summing in time the previous identity we get
\begin{align*}
\|u\|_{\ell^p(h\Z^N\times \mathcal{T}_\tau)}^p&= \sum_{j=0}^\infty (\|u(\cdot,t_{j+1})\|_{\ell^1(h\Z^N)} - \|u(\cdot,t_{j})\|_{\ell^1(h\Z^N)}) \leq \lim_{J\to+\infty}\|u(\cdot,t_{J})\|_{\ell^1(h\Z^N)} \leq C.
\end{align*}
that is, $u\in\ell^p(h\Z^N\times \mathcal{T}_\tau)$. Thus, $u$ must be zero by Lemma \ref{no globales}, which is a contradiction with the fact that $u$ is nontrivial.
\end{proof}

\begin{proof}[Proof of Theorem \ref{teo.fujita1}]
Let us assume that either \eqref{as:weight4} or \eqref{as:weight5} hold. Lemma \ref{lem:simbolp} gives us that in both cases the Fourier symbol satisfies hypothesis \eqref{eq:bddabove} and \eqref{eq:bddabovebis}. On the other hand, Lemmas \eqref{lem:groeigRsecond} and \eqref{lem:decayeigfrac} implies
$$
\lambda_{h,B_R}\le C R^{2s}, \qquad s\in(0,1].
$$
Then, our result  follows by Lemmas \ref{hay globales}, \ref{todobum} and Theorem \ref{fujitabum}.
\end{proof}

\section{Convergence properties}\label{sec:convprop}

In this Section we prove Theorems \ref{thm:convergencebeforeblowup} (convergence of the solutions) and \ref{thm:conbwtimes} (convergence of the blow-up time). We started proving the convergence of the solutions, as usual to prove the convergence theorems we assume that the solution of the continuous problem is regular enough to obtain consistency.

\begin{proof}[Proof of Theorem \ref{thm:convergencebeforeblowup}]
We define the error functions
\begin{align*}
    E(x_\alpha,t_j)=w(x_\alpha,t_j)- u(x_\alpha,t_j) \quad \textup{and} \quad  E(t_j)= \sup_{x_\alpha\in h\mathbb{Z}^N}|E(x_\alpha,t_j)|.
\end{align*}
By the definitions of $\Lambda_1$ and $\Lambda_2$ given in \eqref{Lambdas}, we can subtract the equations for $w$ and $u$ to get
\[
\partial_{\tau_j} E(x_\alpha,t_{j})= L_h E(x_\alpha,t_j) + (w^p(x_\alpha,t_j)-u^{p}(x_\alpha,t_j))+ \Lambda_1(x_\alpha,t_j)+ \Lambda_2(x_\alpha,t_j),
\]
which, after reordering, yields
\begin{align*}
E(x_\alpha,t_{j+1})=&E(x_\alpha,t_{j})\bigg(1- \tau_j \sum_{\beta\in \mathbb{Z}\setminus\{0\}}\omega(\beta,h)\bigg) + \tau_j\sum_{\beta\in \mathbb{Z}\setminus\{0\}}E(x_\alpha+x_\beta,t_j)\omega(\beta,h) \\
&+\tau_j(w^p(x_\alpha,t_j)-u^{p}(x_\alpha,t_j))+ \tau_j\Lambda_1(x_\alpha,t_j)+ \tau_j\Lambda_2(x_\alpha,t_j).
\end{align*}
Let us consider the above identity for for every $t_j \in [0,\tilde{T}]$ with
\begin{equation}\label{eq:timegood}
\tilde{T}= \max\{t_j \in [0,T-\rho]\,:\, E(t_j)\leq 1 \}.
\end{equation}
By \eqref{CFL1}, we can take absolute values, take supremums and use the triangle inequality to get that, for every $t_j\in [0,T-\rho]$, we have
\begin{align*}
E(t_{j+1})\leq& E(t_{j})\bigg(1- \tau_j \sum_{\beta\in \mathbb{Z}\setminus\{0\}}\omega(\beta,h)\bigg) + \tau_jE(t_j) \sum_{\beta\in \mathbb{Z}\setminus\{0\}}\omega(\beta,h)+ \tau_j\sup_{x_\alpha\in h \mathbb{Z}^N}|w^p(x_\alpha,t_j)-u^{p}(x_\alpha,t_j)| \\
&+\tau_j \sup_{x_\alpha\in h \mathbb{Z}^N}|\Lambda_1(x_\alpha,t_j)|+ \tau_j\sup_{x_\alpha\in h \mathbb{Z}^N}|\Lambda_2(x_\alpha,t_j)|\\
\leq& E(t_j) + \tau_jM_j E(t_j) + \tau_j R_{h,\tau},
\end{align*}
with
$
M_j=p\max\{\sup_{x_\alpha\in h \mathbb{Z}^N}|w(x_\alpha,t_j)|,\sup_{x_\alpha\in h \mathbb{Z}^N}|u(x_\alpha,t_j)|\}^{p-1}$ and $R_{h,\tau}=\|w\|_{\mathcal{H}_{T-\rho}}(\varrho_1(h)+\varrho_2(\tau))$.
Note that, on one hand, $\sup_{x_\alpha\in h \mathbb{Z}^N}|w(x_\alpha,t_j)| \leq \|w\|_{L^\infty(\R^N\times[0,T-\rho])}$,
and on the other hand, by \eqref{eq:timegood}, \[\sup_{x_\alpha\in h \mathbb{Z}^N}|u(x_\alpha,t_j)|\leq \sup_{x_\alpha\in h \mathbb{Z}^N}|w(x_\alpha,t_j)|+ \sup_{x_\alpha\in h\Z^N}|E(x_\alpha,t_j)| \leq \|w\|_{L^\infty(\R^N\times[0,T-\rho])}+1.
\]
Thus, $M_j \leq p (\|w\|_{L^\infty(\R^N\times[0,T-\rho])}+1)^{p-1}:=K$
and $\partial_{\tau_j}E(t_{j})\leq K E(t_j) + R_{h,\tau}$.
It is easy to check that
\[
E(t_j)\leq \frac{R_{h,\tau}}{K} (e^{Kt_j}-1) \leq R_{h,\tau} e^{K(T-\rho)}.
\]
using that the explicit solution of the ODE problem $y'(t)= K y(t)+R_{h,\tau}$ with $y(0)=0$ is given by $y(t)=\frac{R}{K} (e^{Kt}-1)$, which is a convex function, and thus, $\partial_{\tau_j} y(t_j) \geq y'(t_j)$. Hence, we have arrived to the estimate
\[
E(t_j) \leq e^{p (\|w\|_{L^\infty(\R^N\times[0,T-\rho])}+1)^{p-1}(T-\rho)} \|w\|_{\mathcal{H}_{T-\rho}} (\varrho_1(h)+\varrho_2(\tau)).
\]
Finally, we note that we can always choose $h$ and $\tau$ small enough such that $E(t_j)\leq 1$, and thus, $\tilde{T}=\max\{t_j \in [0,T-\rho]\}$, which concludes the prove.
\end{proof}

Using the above convergence result, we can prove the convergence of the discrete blow-up time to the continuous one, assuming some upper and lower bounds on the blow-up times.

\begin{lemma}\label{lem:convbutimes}Assume \eqref{as:omega1}, \eqref{as:consistency}  and \eqref{CFL1}. Let $w$ be a blow-up solution of \eqref{1.1} with blow-up time $T$. Assume $w\in \mathcal{H}_{T-\rho}$ for all $\rho\in(0,T)$. Let also $u$ be a blow-up solution of \eqref{eq-discreta} with  blow-up  time $T_{h,\tau}$. Additionally, assume that there exist two constants $C_1$ and $C_2$ independent of $h$ and $\tau$ such that
\begin{equation}\label{eq:uplowbound}
\|w(\cdot,t)\|_{L^\infty(\R^N)} \geq C_1(T-t)^{-\frac{1}{p-1}} \quad \textup{and} \quad \|u(\cdot,t_j)\|_{\ell^\infty(h\Z^N)} \leq C_2(T_{h,\tau}-t_j)^{-\frac{1}{p-1}}.
\end{equation}
Then, $\lim_{h,\tau\to0^+} T_{h,\tau}=T$.
\end{lemma}
\begin{proof}
Note that, as consequence of Theorem \ref{thm:convergencebeforeblowup}, $u$ cannot blow up before time $T$. Let $\rho>0$ and $t_J=\max\{t_j\in[0,T-\rho]\}$. Note that, since $t_{J+1}=t_J+\tau_J$, we have that $t_J\leq T-\rho<t_{J+1}$ and thus, $|(T-\rho)-t_J|\leq |t_{J+1}-t_{J}|=\tau_J$ (that is $|(T-\rho)-t_J|\to 0$ as $\tau\to0^+$). Therefore, for $\tau$ small enough, we have
\begin{align*}
    |T-T_{h,\tau}|&\leq |T-t_J|+|T_{h,\tau}-t_J|\leq \rho+ |(T-\rho)-t_J|+|T_{h,\tau}-t_J|\leq \rho+ \tau_J+|T_{h,\tau}-t_J|.
\end{align*}
We note now that, by \eqref{eq:uplowbound} and \eqref{CFL1},
\[
\|w(\cdot,t_J)\|_{L^\infty(\R^N)} \geq \frac{C_1}{(T-t_J)^{\frac{1}{p-1}}} \geq \frac{C_1}{(\rho+\tau_J)^{\frac{1}{p-1}}}
 \geq \frac{C_1}{(\rho+\tau)^{\frac{1}{p-1}}}
 \quad \textup{and} \quad |T_{h,\tau}-t_J| \leq \frac{C_2^{p-1}}{\|u(\cdot,t_J)\|_{L^\infty(\R^N)}^{p-1}}.
\]
By Theorem \ref{thm:convergencebeforeblowup}, we also have
\begin{align*}
    \|u(\cdot,t_J)\|_{\ell^\infty(h\Z^N)}&\geq\|w(\cdot,t_J)\|_{\ell^\infty(h\Z^N)}- \|u(\cdot,t_J)-w(\cdot,t_J)\|_{\ell^\infty(h\Z^N)} \\
    &\geq  \frac{C_1}{(\rho+\tau)^{\frac{1}{p-1}}} - e^{p (\|w\|_{L^\infty(\R^N\times[0,T-\rho])}+1)^{p-1}(T-\rho)} \|w\|_{\mathcal{H}_{T-\rho}} (\varrho_1(h)+\varrho_2(\tau)) .
\end{align*}
Thus,
$
\limsup_{h,\tau\to 0^+} |T-T_{h,\tau}|\le \rho+\left(C_2/C_1\right)^{p-1}\rho^{\frac{1}{p-1}}.
$
Since $\rho$ is arbitrary, the proof of convergence of blow-up times is completed.
\end{proof}

\begin{proof}[Proof of Theorem \ref{thm:conbwtimes}]
We note that, from Lemma \ref{lem:convbutimes}, we only need to show that \eqref{eq:uplowbound} holds. The blow-up rate for the solution $w$ of \eqref{1.1} follows by comparison with the planar solution $W(x,t) = \left((p-1)(T-t)\right)^{-\frac{1}{p-1}}$, where $T$ is the blow-up time of $w$.
Indeed, if we assume that there exists $t\in(0,T)$ such that
$
\|w(\cdot,t_0)\|_{L^\infty(\R^N)} < \left((p-1)(T-t_0)\right)^{-\frac{1}{p-1}},
$
then, it also holds
$
\|w(\cdot,t_0)\|_{L^\infty(\R^N)} < \left((p-1)(T+\varepsilon-t)\right)^{-\frac{1}{p-1}}
$
for some $\varepsilon$ small enough. But then, by comparison, this implies that
$
\|w(\cdot,T)\|_{L^\infty(\R^N)} < \left((p-1)\varepsilon\right)^{-\frac{1}{p-1}},
$
which is a contradiction with the fact that $T$ is the blow-up time for $w$. Thus, there exists $C_1=C_1(p)>0$ such that
\[
\|w(\cdot,t)\|_{L^\infty(\R^N)} \geq C_1(T-t)^{-\frac{1}{p-1}}
\]
The proof of the upper bound for the blow-up rate for the solution $u$ of \eqref{eq-discreta} requires a more refined technique. We present it now for the two different hypothesis given in Theorem \ref{thm:conbwtimes}.
%

First, let us assume the hypothesis \eqref{lemconbwt-item1}, that is, there exists  $h_0\in(0,1)$ such that $\sup_{h\in(0,h_0)}\sum_{\beta\in \mathbb{Z}^N\setminus\{0\}} \omega(\beta,h)<+\infty$. We can proceed as in the proof of Lemma \ref{lem-bum} to get
\[
\partial_{\tau_j} \|u(\cdot,t_j)\|_{\ell^\infty(h\mathbb{Z}^N)} \geq - \sup_{h\in(0,h_0)}\|\omega(\beta,h)\|_{\ell^1(\mathbb{Z}^N \setminus\{0\})} \|u(\cdot,t_j)\|_{\ell^\infty(h\mathbb{Z}^N)} + \|u(\cdot,t_j)\|_{\ell^\infty(h\mathbb{Z}^N)}^p.
\]
Then, there exists $\overline{t}$ (independent on $h$) such that for all $t_j\geq \overline{t}$ we have that
$
\partial_{\tau_j} \|u(\cdot,t_j)\|_{\ell^\infty(h\mathbb{Z}^N)} \geq \frac{1}{2}\|u(\cdot,t_j)\|_{\ell^\infty(h\mathbb{Z}^N)}^p$.
Thus, we can proceed as in the proof of Lemma \ref{lem-bum} to get that for $\tau\le 1$
\[
\|u(\cdot,t_j)\|_{\ell^\infty(h\mathbb{Z}^N)}  \leq \frac{\tau (1+\frac12 \tau)^{p-1}}{(1+\frac12 \tau)^{p-1}-1}(T_{h,\tau}-t_j)^{\frac{-1}{p-1}}
\leq \frac{3^{p-1}}{3^{p-1}-2^{p-1}}(T_{h,\tau}-t_j)^{\frac{-1}{p-1}},
\]
where in the last inequality we use that the function $\tau\to \frac{\tau (1+\frac12 \tau)^{p-1}}{(1+\frac12 \tau)^{p-1}-1}$ is increasing.

Assume now that hypothesis \eqref{lemconbwt-item2} holds, that is, $
L_h \varphi(x_\alpha) + (1-\veps) \varphi^p(x_\alpha)\geq 0$  for all   $h\in(0,h_0)$  and  $x_\alpha\in h\mathbb{Z}^N$. Consider the function
\[
G(x_{\alpha},t_j) = \partial_{\tau_j} u(x_\alpha,t_j)-\veps u^p(x_\alpha,t_j).
\]
By assumption, $G(x_{\alpha},0)\ge0$. Moreover,
$$
\left(\partial_{\tau_j}  -L_h\right) G(x_\alpha,t_j) = (1-\varepsilon) \partial_{\tau_j} u^p(x_\alpha,t_j) +\varepsilon L_h u^p(x_\alpha,t_j).
$$
By convexity of the power $p$ function, this implies
\begin{align*}
  \left(\partial_{\tau_j}  -L_h\right)G(x_\alpha,t_j) &\geq  (1-\varepsilon)p u^{p-1}(x_\alpha,t_j) \partial_{\tau_j} u(x_\alpha,t_j) +\varepsilon p u^{p-1}(x_\alpha,t_j) L_h u(x_\alpha,t_j)\\
  &=p u^{p-1}(x_\alpha,t_j)\left(\partial_{\tau_j} u(x_\alpha,t_j) -\veps\left(\partial_{\tau_j}- L_h \right)u(x_\alpha,t_j)\right)\\
  &= p u^{p-1}(x_\alpha,t_j)\left(\partial_{\tau_j} u(x_\alpha,t_j) -\veps u^p(x_\alpha,t_j)\right)= p u^{p-1}(x_\alpha,t_j) G(x_\alpha,t_j).
\end{align*}
Thus, by comparison principle, $G(x_\alpha,t_j)\geq0$ for all $x_\alpha\in h\Z^N$ and $t_j\geq0$. This implies, by the definition of $G$, that
$
\partial_{\tau_j} u(x_\alpha,t_j)\geq\veps u^p(x_\alpha,t_j)$.
As before, this implies that
\[
\|u(\cdot,t_j)\|_{\ell^\infty(h\mathbb{Z}^N)}  \leq \frac{\tau (1+\veps \tau)^{p-1}}{(1+\veps \tau)^{p-1}-1}(T_{h,\tau}-t_j)^{\frac{-1}{p-1}}
\leq \frac{(1+\veps)^{p-1}}{(1+\veps)^{p-1}-1}(T_{h,\tau}-t_j)^{\frac{-1}{p-1}},
\]
for all $\tau\leq 1$.
\end{proof}

We conclude this section by completing the proof of Remark \ref{rem:datagood}.

\begin{proof}[Proof of Remark \ref{rem:datagood}]
Indeed, it is clear that
$
L_h u_0(x_\alpha)\ge 0$ for $|x_\alpha|\ge R$.
So in this region we get the desired inequality. Moreover, by the consistence of the operator,
$
L_h u_0 =L u_0+\tau_h+\widetilde\tau_h,
$
where $\tau_h$ and $\widetilde \tau_h$ are the truncation errors of the local and nonlocal parts, respectively. Observe that both truncation errors are bounded and go to zero as $h\to 0^+$.

Now, we consider the region $R-\delta<|x|<R$. Since, $\partial^4_r u_0\ge0$   we get that $\tau_h\ge0$. On the other hand, the continuity of the operator $L$, and the fact that for every $x_0\in \partial B_R(0)$ we have
$$
L u_0 (x_0)= \int_{\R^N} u_0(x_0+y) d\mu(y) \ge \min_{x\in\partial B_R(0)} \int_{\R^N} u_0(x+y) d\mu(y)=A>0,
$$
implies that for $h\le h_0$ small enough $L_h u_0(x)\ge L u_0(x)+\widetilde \tau_h\ge A/2$.

Finally in $|x|<R-\delta$, we get $u_0(x)> B$, then for $h$ small we have $\varepsilon u_0^p+\tau_h+\widetilde \tau_h \ge\frac\varepsilon2 B^p $. Thus,
$
L_h u_0(x)+(1-\varepsilon) u^p =
L u_0(x)+(1-2 \varepsilon) u^p +\varepsilon u^p +\tau_h+\widetilde \tau_h \ge0.
$
\end{proof}

\section{Numerical Experiments}\label{sec:numexp}
In this section, we test our numerical scheme:
\begin{equation*}
\left\{
\begin{aligned}
\partial_{\tau_j} u(x_\alpha,t_j) &= L_h u(x_\alpha,t_j) + u^p(x_\alpha,t_j), && x_\alpha \in h\mathbb{Z}^N,\; t_j \in (0,T), \\
u(x_\alpha,0) &= u_0(x_\alpha), && x_\alpha \in h\mathbb{Z}^N,
\end{aligned}
\right.
\end{equation*}
on interesting special cases in one spatial dimension.  
Our aim is to verify numerically some of our theoretical results, such as the existence of blow-up solutions (Theorem~\ref{teo.fujita1}\eqref{teo.fujita1-blowup}), the blow-up rates (Theorem~\ref{tasas}), the convergence of blow-up times (Theorem~\ref{thm:conbwtimes}), and the existence of global solutions (Theorem~\ref{teo.fujita1}\eqref{teo.fujita1-global}).

Since we have also obtained results on the asymptotic behaviour of the discrete diffusion equation:
\begin{equation*}
\left\{
\begin{aligned}
\partial_{\tau} z(x_\alpha,t_j) - L_h z(x_\alpha,t_j) &= 0, && x_\alpha \in h\mathbb{Z}^N,\; t_j \in \tau\mathbb{N}, \\
z(x_\alpha,0) &= u_0(x_\alpha), && x_\alpha \in h\mathbb{Z}^N,
\end{aligned}
\right.
\end{equation*}
we also present some experiments regarding Theorem~\ref{comportamiento lineal} and Remark~\ref{comportamiento}.

All experiments have been run for $h$ and $\tau$ satisfying \eqref{CFL1}. Since our schemes are naturally posed in the whole space $\mathbb{R}^d$, to numerically compute the solutions we restrict the grids to a sufficiently large bounded spatial domain and set the numerical solution equal to zero outside.



We have conducted experiments using local, zero-order nonlocal, and fractional (nonlocal differential) diffusion operators.  

\begin{itemize}
  \item In the \textbf{local case}, we chose the Laplacian operator:
  \[
  L\phi(x) = \Delta\phi(x).
  \]

  \item In the \textbf{nonlocal zero-order case}, we used the integral operator:
  \[
  L\phi(x) = \int_{-\infty}^{\infty} (\phi(x+y) - \phi(x)) e^{-|y|^2} \, dy.
  \]

  \item In the \textbf{nonlocal differential case}, we used a multiple of the fractional Laplacian defined as
  \[
  L\phi(x) = \textup{P.V.} \int_{|y|>0} \frac{\phi(x+y) - \phi(x)}{|y|^{1+2s}} \, dy, \quad \textup{with} \quad s = \frac{1}{2}.
  \]
\end{itemize}

In each case, we applied the discretization schemes described in Remark~\ref{ejemplos}.

Finally, the initial data used for all experiments was
\[
u_0(x) = \frac{9}{10} (1 - |x|^2)_+,
\]
where $(\cdot)_+$ denotes the positive part.



\subsection{Existence of blow-up solutions}
We present here numerical experiments corresponding to Theorem \ref{teo.fujita1}\eqref{teo.fujita1-blowup}, showing the performance of our numerical scheme in capturing the blow-up phenomenon. We have always chosen $p$ according to the hypotheses of the theorem. As can be seen in Figure \ref{fig:blow-up_solutions}, the blow-up occurs in very different ways for our three examples, yet our scheme is able to capture them accurately. We note that, in order to properly observe the blow-up phenomenon, the $z$-axis in this figure is presented on a logarithmic scale.


\begin{figure}[h!]
  \centering
  \begin{subfigure}[b]{0.32\textwidth}
    \includegraphics[width=\textwidth]{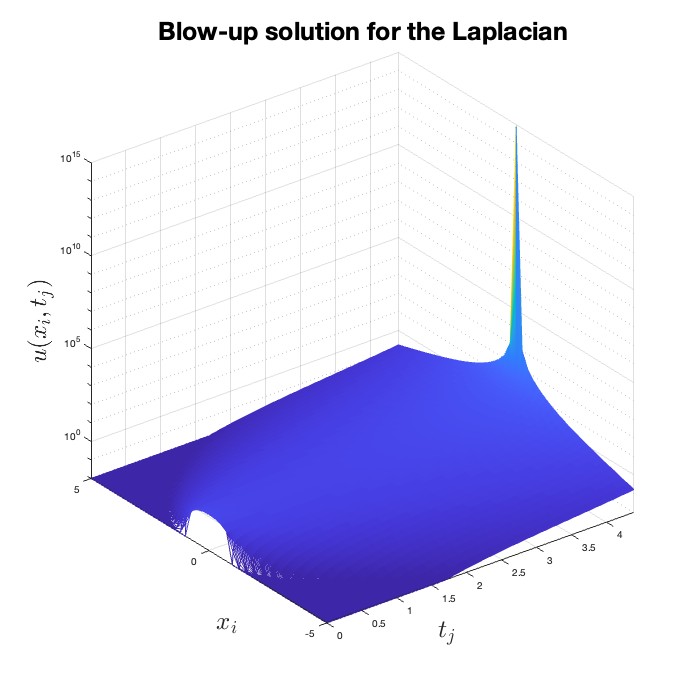}
    \subcaption{$p=2$}
  \end{subfigure}
  \begin{subfigure}[b]{0.32\textwidth}
    \includegraphics[width=\textwidth]{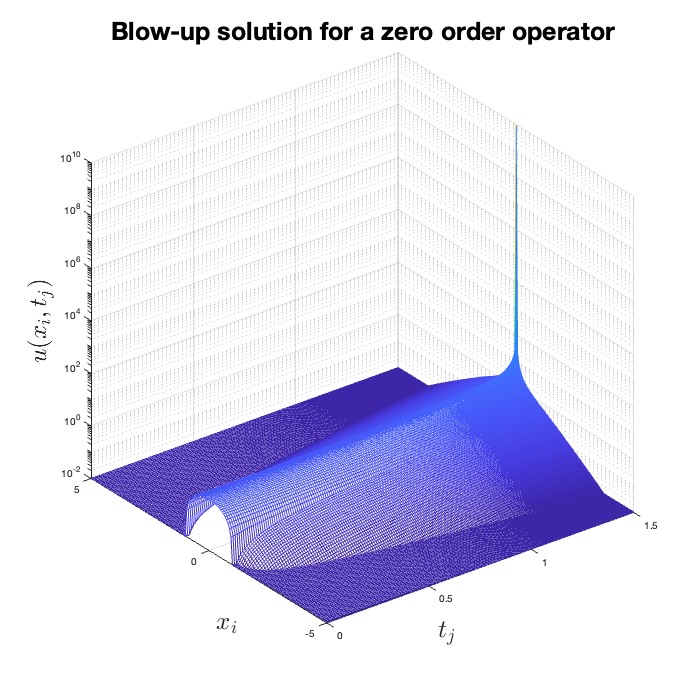}
    \subcaption{$p=2.5$}
  \end{subfigure}
  \begin{subfigure}[b]{0.32\textwidth}
    \includegraphics[width=\textwidth]{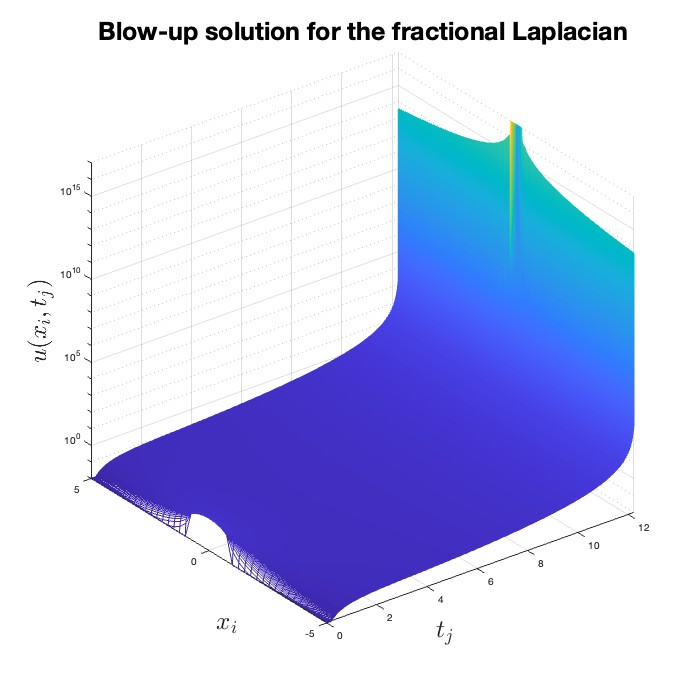}
    \subcaption{$p=1.5$ and $s=0.5$}
  \end{subfigure}
  \caption{Comparison of blow-up solutions for different operators.}
  \label{fig:blow-up_solutions}
\end{figure}

\subsection{Blow-up rates} 
In the presence of blow-up solutions, Theorem \ref{tasas} ensures that  
\[
C_1 (T_{h,\tau}-t_j)^{\frac{-1}{p-1}} \le \|u(\cdot,t_j)\|_{\ell^\infty(h\mathbb{Z}^N)} \le
C_2 (T_{h,\tau}-t_j)^{\frac{-1}{p-1}},
\]  
for $t_j$ large enough. We empirically observe this behavior in Figure \ref{fig:blowuprates}. To properly visualize it, the left figure has both axes in logarithmic scale, while in the right one only the $y$-axis is shown in logarithmic scale.

\begin{figure}[h!]
\includegraphics[width=0.8\textwidth]{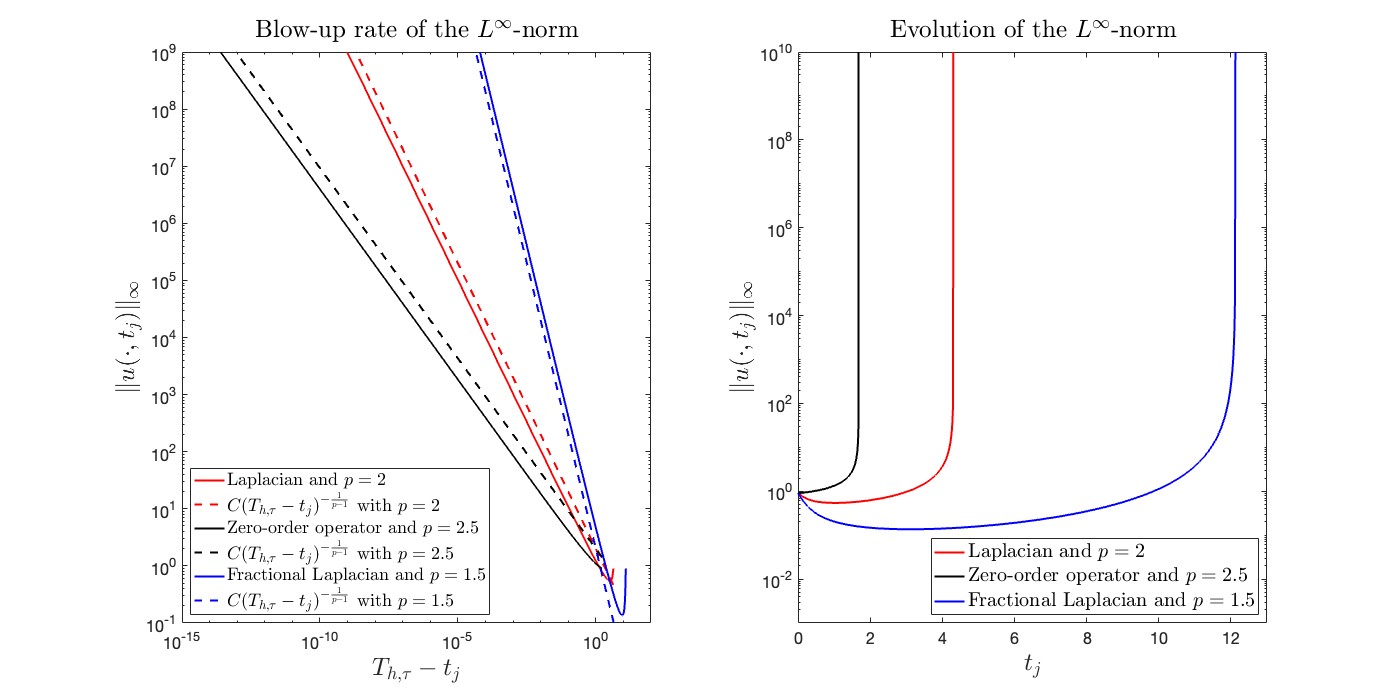}
\caption{Blow-up rates and the evolution of $\|u(\cdot,t_j)\|_\infty$.}
\label{fig:blowuprates}
\end{figure}

\subsection{Convergence of blow-up times} 
Again, in the presence of blow-up solutions, Theorem~\ref{thm:conbwtimes} ensures that the discrete blow-up time obtained from the numerical scheme converges to the actual blow-up time of the continuous problem as the mesh is refined. To provide evidence of this behavior, we present numerical results that support this theoretical prediction in Figure~\ref{fig:convblowuptimes}.

Since, to the best of our knowledge, there are no known explicit blow-up solutions to the nonlinear problem \eqref{1.1} (aside from the trivial ones that arise from the corresponding ordinary differential equation) we adopt a practical strategy to study convergence. Specifically, we compute a reference blow-up time using a highly refined spatial grid with mesh size \( h = h_0 = 2^{-7} \). We then compare this reference value with the blow-up times obtained on coarser grids corresponding to \( h = 2^{-i} \), for \( i = 1,2,3,4,5 \). This comparison allows us to assess how the numerical blow-up time approaches the continuous one as the discretization becomes finer.

To clearly visualize the convergence behavior and estimate the convergence rate, both axes in Figure~\ref{fig:convblowuptimes} are plotted on a logarithmic scale.

\begin{figure}[h!]
\includegraphics[width=0.8\textwidth]{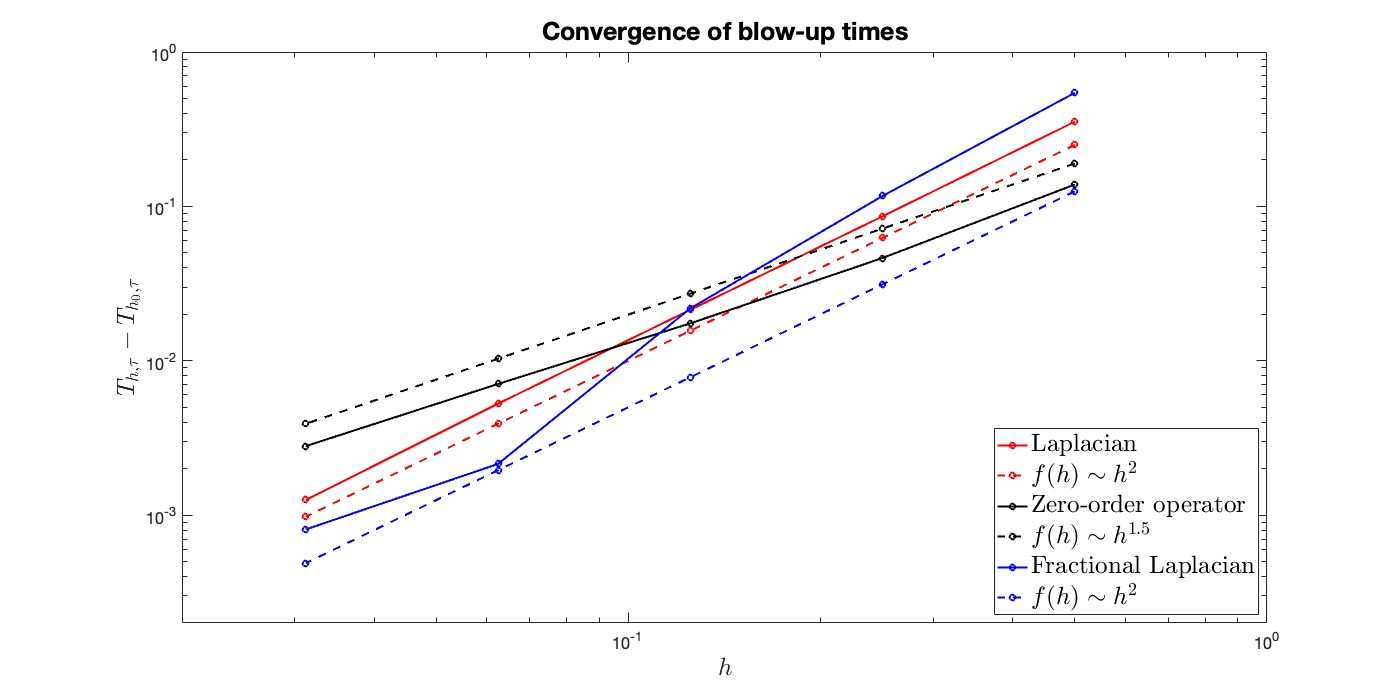}
\caption{Convergence of blow-up times.}
\label{fig:convblowuptimes}
\end{figure}

\subsection{Global solutions} In Figure~\ref{fig:globalsol}, we illustrate the existence of global  solutions as established by Theorem~\ref{teo.fujita1}\eqref{teo.fujita1-global}.
\begin{figure}[h!]
  \centering
  \begin{subfigure}[b]{0.32\textwidth}
    \includegraphics[width=\textwidth]{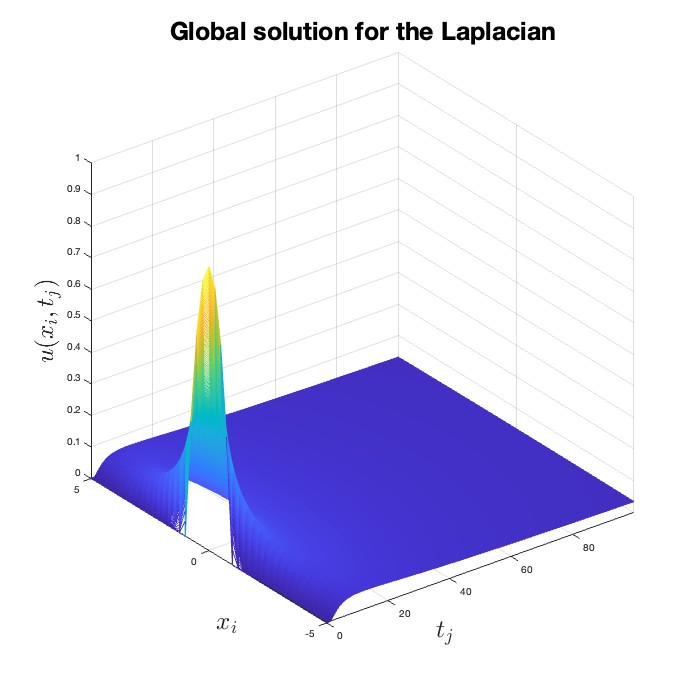}
    \subcaption{$p=5$}
  \end{subfigure}
  \begin{subfigure}[b]{0.32\textwidth}
    \includegraphics[width=\textwidth]{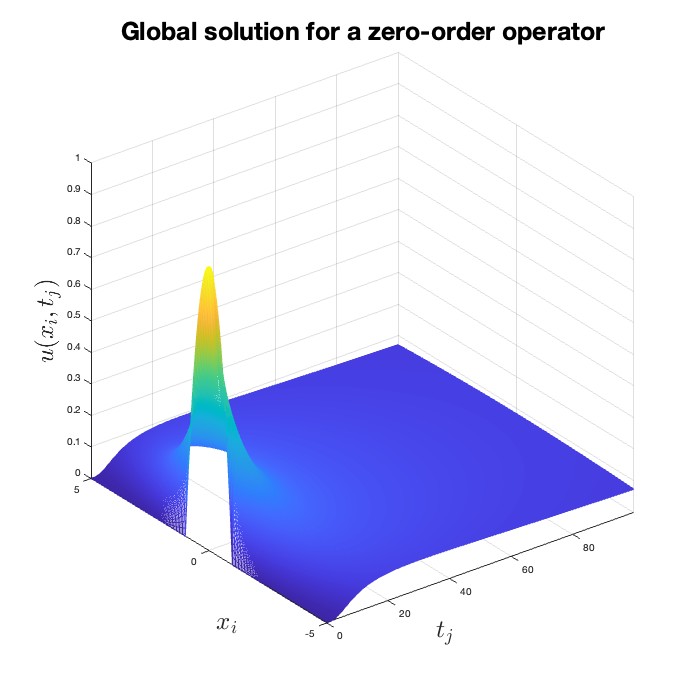}
    \subcaption{$p=5$}
  \end{subfigure}
  \begin{subfigure}[b]{0.32\textwidth}
    \includegraphics[width=\textwidth]{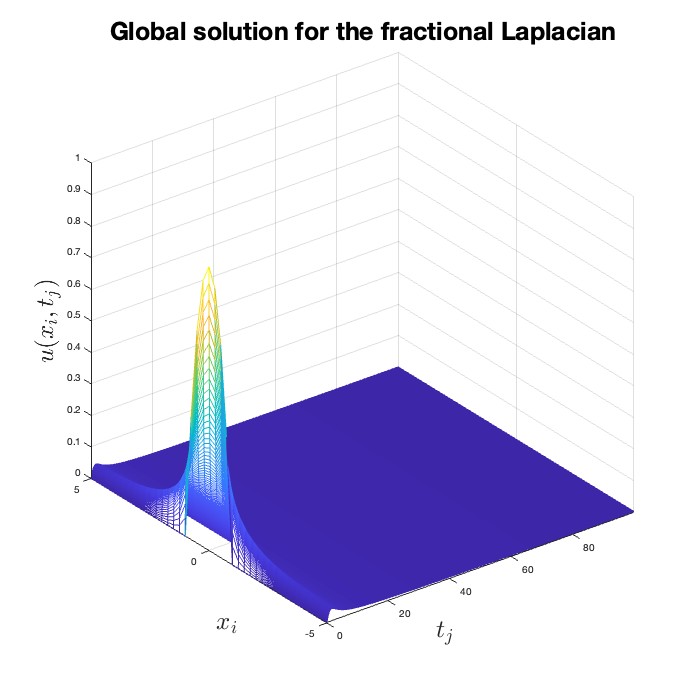}
    \subcaption{$p=5$ and $s=0.5$}
  \end{subfigure}
  \caption{Comparison of global solutions for different operators}
  \label{fig:globalsol}
\end{figure}

\subsection{Asymptotic behaviour of fully discrete discrete diffusion equations}
Finally, we present some experiments illustrating the asymptotic behavior of the discrete diffusion equation \eqref{eq:asbehaintro}, for a fixed spatial step size \( h = 0.5 \). In Figure~\ref{fig:asybeha}, we show numerical evidence supporting the convergence estimate established in Theorem~\ref{comportamiento lineal}.

In this particular cases, the constant \( C \) appearing in the asymptotic estimate is equal to zero, due to the fact that the Fourier symbols associated with the discrete operators satisfy the condition
\[
\lim_{\xi \to 0} \frac{m(\xi)}{|\xi|^{2s}} = -K,
\]
as required by the assumptions stated in Remark~\ref{comportamiento}. This validation confirms that the long-time behaviour of the discrete solution matches the theoretical predictions, even for moderate values of \( h \).

\begin{figure}[h!]
\includegraphics[width=0.8\textwidth]{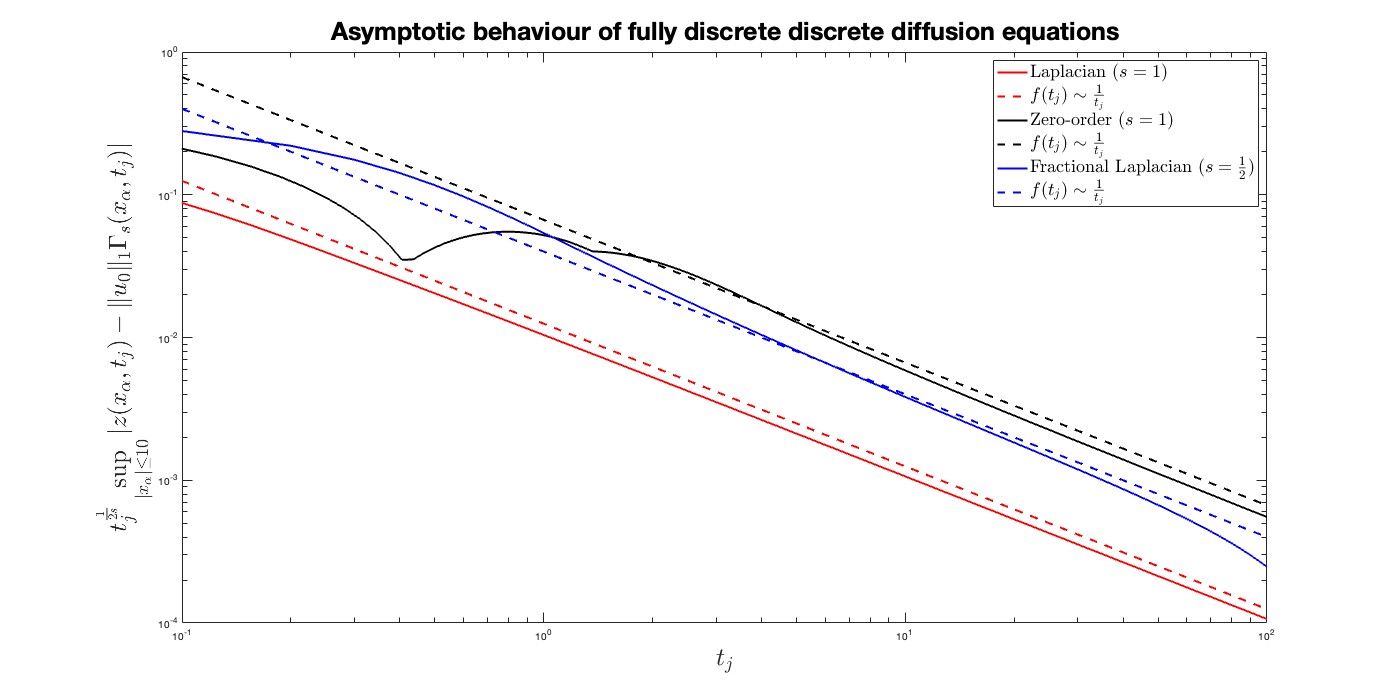}
\caption{Asymptotic behaviour of fully discrete difussion equations.}
\label{fig:asybeha}
\end{figure}

\appendix

\section{An appropriate choice of the time step}\label{sec:choicetau}
Due to the possible blow-up behaviour of \eqref{eq-discreta}, we need to choose an adaptive time step. In order to figure out the appropriate time step, let us consider the solution \eqref{eq-discreta} with constant initial data $\varphi(x_\alpha)=Y^0\in (0,\infty)$ for all $\alpha\in\Z^N$. Then, the solution is given by $U(x_\alpha,t_j)=Y(t_j)$ for all $\alpha\in \Z^N$, where $Y$ satisfies
\begin{equation}\label{eq-discreta-ODE}
\left\{
\begin{array}{ll}
Y(t_{j+1})=Y(t_j)+\tau_j Y^p(t_j),\quad & t_j\in(0,T), \\
Y(t_0)=Y^0. &
\end{array}\right.
\end{equation}
Note that \eqref{eq-discreta-ODE} is just an explicit Euler scheme associated to the ODE $y'(t)=y(t)^p$, whose solution is given by
\begin{align*}
 y(t)= C_p (T^{\textup{b}}-t)^{-\frac{1}{p-1}}, \quad \textup{where} \quad T^{\textup{b}}= (Y^0)^{1-p}/(p-1) \quad \textup{and} \quad C_p=(p-1)^{-\frac{1}{p-1}},
\end{align*}
for all $t\in[0,T^{\textup{b}})$, which blows up in finite time $T^{\textup{b}}$.

At this point, we realize that, for an specific choice of $\tau_j$, problem \eqref{eq-discreta-ODE} can be solved explicitly. More precisely, let us  fix $\tau>0$ and choose $\tau_j\coloneqq \tau Y^{1-p}(t_j)$.
Then, the solution of \eqref{eq-discreta-ODE} is given by
\begin{align}\label{sol-disc-ODE}
   Y(t_j)=(1+\tau)^{j}Y^0.
\end{align}
We can also compute explicitly the time $t_j$ for $j>0$:
\[
t_j=\sum_{k=0}^{j-1}\tau_k = \sum_{k=0}^{j-1}\tau Y^{1-p}(t_k) = \tau (Y^0)^{1-p} \sum_{k=0}^{j-1} ((1+\tau)^{{1-p}})^{k} =  \tau (Y^0)^{1-p}  \frac{1-((1+\tau)^{{1-p}})^{j}}{1-((1+\tau)^{{1-p}})}.
\]
In particular, this shows that \eqref{sol-disc-ODE} also blows up in finite time
$
T^{\textup{b}}_\tau= \lim_{j\to+\infty} t_j=\frac{\tau (Y^0)^{1-p}}{1-(1+\tau)^{{1-p}}}.
$
We observe that, $T^{\textup{b}}_\tau \to T^{\textup{b}}$ as $\tau\to0$:
\[
\lim_{\tau\to0}T^{\textup{b}}_\tau=(Y^0)^{1-p} \lim_{\tau\to0} \frac{\tau }{1-(1+\tau)^{{1-p}}} = (Y^0)^{1-p}\lim_{\tau\to0} \frac{1}{(p-1)(1+\tau)^{{-p}}}=   \frac{(Y^0)^{1-p}}{p-1}= T^{\textup{b}}.
\]
Actually, we can also show convergence of $Y$ to $y$ arbitrarily close to the blow-up time. More precisely, let us write, by direct computations
\[
Y(t_j)=C_{p,\tau} (T_\tau^{\textup{b}}-t_j)^{-\frac{1}{p-1}} \quad \textup{where} \quad T^{\textup{b}}_\tau= \frac{\tau (Y^0)^{1-p}}{1-(1+\tau)^{{1-p}}} \quad \textup{and} \quad C_{p,\tau}=\left(\frac{1-(1+\tau)^{1-p}}{\tau}\right)^{-\frac{1}{p-1}}.
\]
Then, for any $T_0< T^{\textup{b}}$, we have
\begin{align*}
\sup_{t_j\leq T_0}|y(t_j)-Y(t_j)| &\leq |C_p-C_{p,\tau}| \sup_{t_j\leq T_0}|T^\textup{b}-t_j|^{-\frac{1}{p-1}}+ C_{p,\tau} \sup_{t_j\leq T_0}\left|(T^{\textup{b}}-t_j)^{-\frac{1}{p-1}}-(T^{\textup{b}}_\tau-t_j)^{-\frac{1}{p-1}}\right|\\
&\lesssim |C_p-C_{p,\tau}| + |T^b-T^{\textup{b}}_{\tau}|,
\end{align*}
 which converges to $0$ as $\tau \to 0^+$.

\section*{Acknowledgements}
FdT was supported by the Spanish Government through RYC2020-029589-I, PID2021-
127105NB-I00 and CEX2019-000904-S funded by the MICIN/AEI. RF was supported by the Spanish project   PID2023-146931NB-I00 funded by the MICIN/AEI and by Grupo de Investigaci\'on UCM 920894.

\bibliographystyle{abbrv}


\end{document}